\documentclass[12pt,leqno]{article}
\usepackage{amssymb, amscd, amsmath, amsthm}
\usepackage{dsfont}
\allowdisplaybreaks

\def\beq{\begin{equation}}
\def\eeq{\end{equation}}

\theoremstyle{definition}
\newtheorem{definition}{Definition}

\newtheorem{example}{Example}

\theoremstyle{plain}
\newtheorem{theorem}{Theorem}
\newtheorem{claim}{Claim}

\newtheorem{lemma}{Lemma}
\newtheorem{conjecture}{Conjecture}
\newtheorem{corollary}{Corollary}
\newtheorem{observation}{Observation}
\newtheorem{proposition}{Proposition}

\numberwithin{equation}{section}
\numberwithin{proposition}{section}
\numberwithin{definition}{section}
\numberwithin{theorem}{section}
\numberwithin{problem}{section}
\numberwithin{example}{section}
\numberwithin{remark}{section}
\numberwithin{claim}{section}
\numberwithin{fact}{section}
\numberwithin{lemma}{section}
\numberwithin{conjecture}{section}
\numberwithin{corollary}{section}
\numberwithin{observation}{section}

\begin{document}

\title{Old and new applications of Katona's circle}

\author{by Peter Frankl\\
R\'enyi Institute, Budapest, Hungary}

\date{}
\maketitle

\begin{abstract}
The present paper is to honour Gyula Katona, my teacher on the occasion of his 80\textsuperscript{th} birthday.
Its main content is threefold, new proofs of old results (e.g.\ the De Bonis, Katona, Swanepoel Theorem on butterfly-free families), new results obtained via the Katona Circle (e.g. for the sum of the sizes of non-empty cross-intersecting families), the solution for the Katona Circle of some notoriously difficult problems (e.g. Erd\H{o}s Matching Conjecture).
\end{abstract}

\section{Introduction}
\label{sec:1}

It is hard to believe but Gyula Katona is 80 years old.
Let me explain how I joined his seminar back in the fall of 1973.
Those days seminars at E\"otv\"os University were advertised by mostly handwritten pieces of paper on a clipboard.
I had a general interest in mathematics and was trying to find some exciting topics.
Katona's note read as follows.

Suppose that $A_1, A_2, \dots, A_m$ are distinct subsets of a set of $n$ elements.
If no two of them are disjoint then $m \leq 2^{n - 1}$.
If you can prove this then come to my seminar!

With a little thought I could find the proof.
As a naive young student I felt obliged to go to Katona's seminar, ``Extremal families of sets''.

As a matter of fact I went to many other seminars as well.
The reason that I stuck with Katona's seminar is twofold.
The first is that he was a wonderful person, treating us students as friends.
The second is that many of the proofs were very elegant.
A very good example to explain this point is the \emph{Circle Method} which Katona has just invented and which has been
subsequently turned into an extremely useful tool of the topic.
Let me call it \emph{Katona's Circle} throughout this paper.

Before explaining Katona's Circle let us fix some notation and introduce a more general averaging argument.

Let $[n] = \{1,2,\ldots, n\}$ be the standard $n$-element set and $2^{[n]}$ its power set.
For $0 \leq k \leq n$ let ${[n]\choose k}$ be the collection of all $k$-subsets of~$[n]$.
Subsets of $2^{[n]}$ are called families.
If $\mathcal F \subset {[n]\choose k}$, it is said to be $k$-uniform.

Let $S_n$ be the symmetric group acting on $[n]$.
For a permutation $\pi \in S_n$ and a family $\mathcal G \subset {[n]\choose k}$ define $\pi(\mathcal G) = \bigl\{\pi(G) : G \in \mathcal G\bigr\}$ where $\pi(G)$ is the image of $G$ under $\pi$.
That is, for $G = \{a_1, \dots, a_k\}$, $\pi(G) = \bigl\{\pi(a_1), \ldots, \pi(a_k)\bigr\}$.

\begin{lemma}[Katona's Averaging Argument \cite{K4}]
\label{lem:1.1}
Suppose that $\mathcal F, \mathcal G \subset {[n]\choose k}$, then
\beq
\label{eq:1.1}
|\mathcal F| \Bigm/ {n\choose k} \leq \max_{\pi \in S_n}\left\{ \frac{|\pi(\mathcal G) \cap \mathcal F|}{|\mathcal G|}\right\}.
\eeq
\end{lemma}

\begin{proof}
Let us choose a $k$-element set $H$ uniformly at random.
The LHS of \eqref{eq:1.1} is the probability $p$ of $H \in \mathcal F$.
On the other hand for any $G \in \mathcal G$ the probability of $\pi(G) \in \mathcal F$ is again $p = |\mathcal F| \bigm/{n \choose k}$.
Consequently, the expectation of $\bigl|\pi(\mathcal G) \cap \mathcal F\bigr|$ is $p|\mathcal G|$.
Thus $|\mathcal F| \bigm/{n\choose k}$ is the expectation of $\bigl|\pi(\mathcal G) \cap \mathcal F\bigr| \bigm/ |\mathcal G|$.
Since the expectation never exceeds the maximum, \eqref{eq:1.1} follows.
\end{proof}

To make \eqref{eq:1.1} useful one needs to choose $\mathcal G$ relatively large.
Indeed, the upper bound will be never less than $1/|\mathcal G|$.
On the other hand in order to find good bounds for $\bigl|\pi(\mathcal G) \cap \mathcal F\bigr|$, it is better to have $|\mathcal G|$ relatively small.
Katona's Circle might be the perfect answer to this dilemma.

\setcounter{definition}{1}
\begin{definition}
\label{def:1.2}
Let $\pi = (x_1, \ldots, x_n)$ be a permutation of $(1,2,\ldots, n)$.
Think of $x_1, \ldots, x_n$ in a cyclic way, that is, the element after $x_n$ is $x_1$.
For $1 \leq k < n$ and for all $1 \leq r \leq n$ define the \emph{arc} $A_k(r)$ as the $k$-set $\bigl(x_r, x_{r + 1}, \ldots, x_{r + k - 1}\bigr)$
where computation is modulo~$n$.
E.g., $A_3(4) = (x_4, x_5, x_6)$, $A_4(n - 1) = (x_{n - 1}, x_n, x_1, x_2)$.
The elements $x_r$ $(x_{k - 1})$ are called the head (tail) of the arc, respectively.

Note that $\emptyset$ and $[n]$ are not included among the arcs.
\end{definition}

In Section \ref{sec:3} we are going to present Katona's short proof of the Erd\H{o}s--Ko--Rado Theorem \cite{EKR} using
$\mathcal G = \mathcal A(n, k) := \bigl\{A_r(k) : 1 \leq r \leq n\bigr\}$.

Let us use here $\mathcal A(n) = \mathcal A(n, 1) \cup \ldots \cup \mathcal A(n, n - 1)$ to give a short proof of Sperner's Theorem \cite{S} (see Theorem \ref{th:1.6} below).

\begin{definition}
\label{def:1.3}
The family $\mathcal R \subset 2^{[n]}$ is called an \emph{antichain} if $R \subset R'$ never holds for distinct $R, R' \in \mathcal R$.
\end{definition}
\setcounter{theorem}{3}
\begin{theorem}[LYM inequality, \cite{Y}, \cite{B}, \cite{Lu}, \cite{Me}]
\label{th:1.4}
Suppose that $\mathcal R \subset 2^{[n]}$ is an antichain.
Then
\beq
\label{eq:1.2}
\sum_{R \in \mathcal R} 1\Bigm/{n\choose |R|} \leq 1.
\eeq
Moreover, in case of equality $\mathcal R = {[n]\choose \ell}$ for some $0 \leq \ell \leq n$.
\end{theorem}

Note that if $\emptyset \in \mathcal R$ then $\mathcal R$ has no other members.
The same is true in the case $[n] \in \mathcal R$.
Therefore from now on we suppose $0 < |R| < n$ for all $R \in \mathcal R$.

Let us prove a lemma which might be called \emph{Circular Sperner Theorem}.

\setcounter{lemma}{4}
\begin{lemma}
\label{lem:1.5}
\beq
\label{eq:1.3}
|\mathcal R \cap \mathcal A(n)| \leq n
\eeq
with equality if and only if $\mathcal R \cap \mathcal A(n) = \mathcal A(n, \ell)$ for some $1 \leq \ell < n$.
\end{lemma}

\begin{proof}[Proof of the Lemma]
Let $\mathcal R \cap \mathcal A(n) = \{B_1, \ldots, B_q\}$ and let $h(B_i)$ denote the head of the arc $B_i$.
Since $\mathcal R \cap \mathcal A(n)$ is still an antichain, $h(B_i) \neq h(B_j)$ for $1 \leq i < j \leq q$.
This proves $q \leq n$.

Assume next $q = n$.
By symmetry suppose $h(B_i) = x_i$, $1 \leq i \leq n$.
Unless $|B_i|$ is constant, we can find $1 \leq i \leq n$ so that $|B_i| > |B_{i + 1}|$.
However, this implies $B_i \supset B_{i + 1}$, a contradiction.
\end{proof}

To prove Theorem \ref{th:1.4} first note that the probability of $R \in \mathcal A(n)$ is $\frac{|\mathcal A(n, |R|)|}{{n\choose |R|}} = \frac{n}{{n\choose |R|}}$.
Thus the expectation of $|\mathcal R \cap \mathcal A(n)|$ is $n \sum\limits_{R \in \mathcal R} 1 \bigm/{n\choose |R|}$.

Since the expectation never exceeds the maximum, \eqref{eq:1.2} follows.
Suppose that equality holds in \eqref{eq:1.2}.
Then equality must hold in \eqref{eq:1.3} for every (cyclic) permutation $(x_1, \dots, x_n)$.

However, if $R, R' \in \mathcal R$ satisfy $|R| \neq |R'|$ then for an appropriate permutation both $R$ and $R'$ are in $\mathcal A(n)$ whence
$|\mathcal R \cap \mathcal A(n)| \leq n - 1$.
This proves that equality in \eqref{eq:1.2} implies $\mathcal R = {[n]\choose \ell}$ for some $\ell$.\hfill $\square$

\smallskip
Since among the binomial coefficients ${n\choose k}$, $0 \leq k \leq n$ the largest are ${n\choose \lfloor n/2\rfloor}$ and ${n\choose \lceil n /2\rceil}$ (for $n$ even they are the same), Theorem \ref{th:1.4} implies

\setcounter{theorem}{5}
\begin{theorem}[Sperner Theorem \cite{S}]
\label{th:1.6}
Suppose that $\mathcal F \subset  2^{[n]}$ is an antichain.
Then
\beq
\label{eq:1.4}
|\mathcal F| \leq {n\choose \lfloor n/2\rfloor}
\eeq
with equality holding if and only if
$$
\mathcal F = {[n]\choose \lfloor n/2\rfloor} \ \ \text{ or } \ \ \mathcal F = {[n]\choose \lceil n/2\rceil} \
\text{(which coincide for $n$ even)}.
$$
\end{theorem}

\section{The circular Kruskal--Katona Theorem}
\label{sec:2}

One of the cornerstones of extremal set theory is the Kruskal--Katona Theorem.
To state it let us define the \emph{shadow} of a family.

\begin{definition}
\label{def:2.1}
Let $\mathcal F \subset {[n]\choose k}$, $0 < \ell < k$.
The $\ell$\emph{-shadow} $\sigma^{(\ell)}(\mathcal F)$ is defined as $\sigma^{(\ell)}(\mathcal F) = \left\{G \in {[n]\choose \ell} : \exists F \in \mathcal F, G \subset F \right\}$.
\end{definition}

Given $m, k$ and $\ell$ the Kruskal--Katona Theorem determines the minimum of $|\sigma^{(\ell)}(\mathcal F)|$ over all $k$-uniform families $\mathcal F$ satisfying $|\mathcal F| = m$.
To avoid introducing somewhat complicated notation let us only state the following version due to Lov\'asz \cite{L}.
Cf. \cite{F84} for a simple proof of both versions.

\setcounter{theorem}{1}
\begin{theorem}[\cite{Kr}, \cite{K2}, \cite{L}]
\label{th:2.2}
Suppose that $\mathcal F \subset {[n]\choose k}$, $|\mathcal F| = {x\choose k}$, $x \geq k$ a real number.
Then for every $0 < \ell < k$,
\beq
\label{eq:2.1}
|\sigma^{\ell}(\mathcal F)| \geq {x\choose \ell}.
\eeq
\end{theorem}

Note that if $x$ is an integer then ${[x]\choose k}$ shows that \eqref{eq:2.1} is best possible.
Let us note also that $n$ is unimportant for \eqref{eq:2.1}, i.e., one gets the same bound for all $n \geq x$.

On the other hand if we fix $[n]$ then we can define the \emph{shade} $\sigma^{(\ell)}(\mathcal F)$ for $k < \ell < n$ as well:
$$
\sigma^{(\ell)}(\mathcal F) = \left\{H \in {[n]\choose \ell} : \exists F \in \mathcal F, \, F \subset H\right\}.
$$
Let us define further the \emph{immediate shadow} $\partial^- \mathcal F$ and the \emph{immediate shade} $\partial^+ \mathcal F$ by
$$
\partial^- \mathcal F = \sigma^{(k - 1)}(\mathcal F), \ \ \ \partial^+ \mathcal F = \sigma^{(k + 1)}(\mathcal F).
$$

Over 30 years ago the author (\cite{F90}) proved the following

\begin{theorem}[Circular Kruskal--Katona Theorem]
\label{th:2.3}
Suppose that $\mathcal B \subset \mathcal A(n, k)$.
Then
\beq
\label{eq:2.2}
\bigl|\sigma^{(\ell)}\mathcal B \cap \mathcal A(n, \ell)\bigr| \geq \min \bigl\{n, |\mathcal B| + |\ell - k|\bigr\}; \ \ \ 1 \leq \ell, k < n.
\eeq
\end{theorem}

\begin{proof}
Since we are only dealing with sets in $\mathcal A(n)$, with a slight abuse of notation we write $\sigma^{(\ell)}(\mathcal B)$ instead of $\sigma^{(\ell)}(\mathcal B) \cap \mathcal A(n, \ell)$.

If $|\mathcal B| = n$, that is, $\mathcal B = \mathcal A(n, k)$ then $\sigma^{(\ell)}(\mathcal B) = \mathcal A(n, \ell)$ for all $1 \leq \ell < n$.
Suppose $|\mathcal B| < n$ and prove
\beq
\label{eq:2.3}
|\partial^+ \mathcal B| \geq |\mathcal B| + 1.
\eeq
Set $q = |\mathcal B|$.
Let $\mathcal B = \bigl\{A_k(y_1), \ldots, A_k(y_q)\bigr\}$.
Using $q < n$ we may assume $2 \leq y_1 < \ldots < y_q \leq n$.
Obviously $A_{k + 1}(y_j) \in \partial^+ \mathcal B$, $1 \leq j \leq q$.
Also $A_{k + 1}(y_1 - 1) \in \partial^+ \mathcal B$.
These are $q + 1$ distinct $(k + 1)$-arcs which completes the proof of \eqref{eq:2.3}.

The inequality
\beq
\label{eq:2.4}
|\partial^- \mathcal B| \geq |\mathcal B| + 1
\eeq
can be proved in the same way.

There is another way to prove it, namely, considering the family of complements $\mathcal B^c = \{[n] \setminus B : B \in \mathcal B\}$ and using
\beq
\label{eq:2.5}
(\partial^- \mathcal B)^c = \partial^+ \mathcal B^c.
\eeq
Now \eqref{eq:2.2} follows by applying \eqref{eq:2.3} or \eqref{eq:2.4} $|k - \ell|$ times.
\end{proof}

We should mention that \eqref{eq:2.5} shows the rather surprising fact that the sizes of immediate shades and shadows coincide for the circle (unless $k = 1$ or $n - 1$).

For further use let us find a meaningful formula for it.
One can associate a graph $G = G(\mathcal B)$ with $\mathcal B \subset \mathcal A(n, k)$.
First let $G(\mathcal A(n, k))$ be the $n$-cycle where $A_k(x)$ and $A_k(x + 1)$ are joined by an edge $(x = 1,2, \ldots, n)$.
Now $G(\mathcal B)$ is the subgraph of $G(\mathcal A(n, k))$ spanned by~$\mathcal B$.
Unless $\mathcal B = \mathcal A(n, k)$, $G(\mathcal B)$ is the vertex disjoint union of paths and isolated vertices.
Let $\lambda(\mathcal B)$ denote the number of (nonempty) connected components of $G(\mathcal B)$.

\setcounter{proposition}{3}
\begin{proposition}
\label{prop:2.4}
Suppose that $\mathcal B \subset \mathcal A(n,k)$, $|\mathcal B| < n$, $1 \leq k < n$.
Then
\begin{align}
\label{eq:2.6}
|\partial^+ \mathcal B| &= |\mathcal B| + \lambda(\mathcal B) \ \ \text{ unless } \ k = n - 1,\\
\label{eq:2.7}
|\partial^- \mathcal B| &= |\mathcal B| + \lambda(\mathcal B) \ \ \text{ unless } \ k = 1. \hspace*{30mm} \square
\end{align}
\end{proposition}

Unfortunately, Theorem \ref{th:2.3} or the Katona Circle Method do not seem suitable to provide a new proof of the Kruskal--Katona Theorem.
However they permit a short proof of Sperner's Shadow Theorem which served as the main tool of his proof of what is now known as Sperner Theorem.

\setcounter{theorem}{4}
\begin{theorem}[\cite{S}]
\label{th:2.5}
Let $n > k \geq 1$ be integers and $\mathcal F \subset {[n]\choose k}$.
Then \eqref{eq:2.8} and \eqref{eq:2.9} hold.
\begin{align}
\label{eq:2.8}
|\partial^+ \mathcal F| &\geq |\mathcal F| \cdot {n\choose k + 1}\Bigm/{n \choose k},\\
\label{eq:2.9}
|\partial^- \mathcal F| &\geq |\mathcal F| \cdot {n\choose k - 1}\Bigm/{n \choose k}.
\end{align}
\end{theorem}

The proof is almost trivial.
Define $\alpha_k, \alpha_{k + 1}, \alpha_{k - 1}$ as follows.
$\alpha_k = |\mathcal F| \bigm/{n\choose k}$, $\alpha_{k + 1} = |\partial^+\mathcal F|\bigm/{n\choose k + 1}$,
$\alpha_{k - 1} = |\partial^-\mathcal F|\bigm/{n\choose k - 1}$.
Then \eqref{eq:2.8} and \eqref{eq:2.9} claim $\alpha_{k + 1} \geq \alpha_k$ and $\alpha_{k - 1} \geq \alpha_k$, respectively.

As to their intersection with $\mathcal A(n)$, the expected size is $\alpha_k n$, $\alpha_{k + 1} n$ and $\alpha_{k - 1} n$, respectively.

If $|\mathcal F \cap \mathcal A(n)| = n$ then the immediate shade and shadow share this value too.
Otherwise Theorem \ref{th:2.3} implies
$$
|\partial^+ \mathcal F \cap \mathcal A(n)| > |\mathcal F\cap \mathcal A(n)| \ \ \text{ and } \ \
|\partial^- \mathcal F \cap \mathcal A(n)| > |\mathcal F\cap \mathcal A(n)|.
$$
These imply $\alpha_{k - 1} \geq \alpha_k$ and $\alpha_{k + 1} \geq \alpha_k$.\hfill $\square$

\smallskip
Let us mention that formally there is a problem with $\partial^+$ if $k = n - 1$ and $\partial^-$ in the case $k = 1$.
However, if $\mathcal F \neq \emptyset$ then in these cases $\alpha_n = 1$ and $\alpha_0 = 1$ hold respectively.

\section{Intersecting families}
\label{sec:3}

A family $\mathcal F$ is called intersecting if $F \cap F' \neq \emptyset$ for all $F, F'\in \mathcal F$.
Let us state the Erd\H{o}s--Ko--Rado Theorem, one of the most important results in extremal set theory.

\begin{theorem}
\label{th:3.1}
Suppose that $\mathcal F \subset {[n]\choose k}$ is intersecting, $n \geq 2k$.
Then
\beq
\label{eq:3.1}
|\mathcal F| \leq {n - 1\choose k - 1}.
\eeq
\end{theorem}

Actually Katona invented the Circle Method in order to give a nice simple proof of \eqref{eq:3.1}.
To achieve this he proved:

\begin{theorem}[Circular Erd\H{o}s--Ko--Rado Theorem, \cite{K3}]
\label{th:3.2}
Suppose that $\mathcal B \subset \mathcal A(n, k)$ is intersecting, $n \geq 2k$.
Then
\beq
\label{eq:3.2}
|\mathcal B| \leq k.
\eeq
\end{theorem}

\begin{proof}
Choose a set $B_0 \in \mathcal B$.
By symmetry assume $B_0 = \{x_1, \dots, x_k\}$.
Since $\mathcal B$ is intersecting all other members $B \in \mathcal B$ have either their head or tail in $B_0$.
In $\mathcal B \setminus \{B_0\}$ the candidates for head are $x_2, \dots, x_k$ and the candidates for tail are $x_1, \dots, x_{k - 1}$.
These are $2(k - 1)$ candidates.
The point is that for $n \geq 2k$ the $k$-arc with tail $x_i$ and the $k$-arc with head $x_{i + 1}$ are disjoint.
Hence at most one of them is in $\mathcal B \setminus \{B_0\}$.

Consequently,
$$
\hspace*{30mm}
|\mathcal B| = 1 + |\mathcal B\setminus \{ B_0\}| \leq 1 + \frac{2(k - 1)}{2} = k.
\hspace*{30mm}\qedhere
$$
\end{proof}

Daykin \cite{D} gave a short proof of the original Erd\H{o}s--Ko--Rado Theorem based on the Kruskal--Katona Theorem.
Let us use the Circular Kruskal--Katona Theorem to prove an extension of Theorem \ref{th:3.2}.

\setcounter{definition}{2}
\begin{definition}
\label{def:3.3}
Two families $\mathcal F$ and $\mathcal G$ are called cross-intersecting if $F \cap G \neq \emptyset$ for all $F \in \mathcal F$, $G \in \mathcal G$.
\end{definition}

\setcounter{theorem}{3}
\begin{theorem}
\label{th:3.4}
Let $k, \ell \geq 1$ be integers, $k + \ell \leq n$.
Suppose that $\emptyset \neq \mathcal B \subset \mathcal A(n, k)$ and $\emptyset \neq \mathcal C \subset \mathcal A(n, \ell)$ are cross-intersecting.
Then
\beq
\label{eq:3.3}
|\mathcal B| + |\mathcal C| \leq k + \ell.
\eeq
Moreover, for $k + \ell < n$, equality implies that $\lambda(\mathcal B) = \lambda(\mathcal C) = 1$.
\end{theorem}

\begin{proof}
Let us first deal with the case $k + \ell = n$.
Consider the family of complements, $\mathcal B^c = \{[n] \setminus B : B \in \mathcal B\}$.
Since $[n]\setminus B$ has empty intersection with $B$, it cannot be a member of $\mathcal C$.
Thus $\mathcal B^c \cap \mathcal C = \emptyset$.
This implies
$$
|\mathcal B| + |\mathcal C| = |\mathcal B^c| + |\mathcal C| \leq n \ \ \text{ as desired.}
$$
Now let us apply reverse induction on $k + \ell$.
Supposing that \eqref{eq:3.3} holds for the pairs $(k + 1, \ell)$, $(k, \ell + 1)$, let us prove it for $(k, \ell)$.
Since $k + \ell \leq n - 1$ and both $\ell$ and $k$ are positive, $k < n - 1$ and $\ell < n - 1$ follow.
In view of \eqref{eq:2.3},
$$
|\partial^+ \mathcal C| \geq |\mathcal C| + 1 \ \ \text{ and } \ \ |\partial^+\mathcal B| \geq |\mathcal B| + 1
$$
follow.

Since $(\mathcal B, \partial^+ \mathcal C)$ and $(\partial^+\mathcal B, \mathcal C)$ are also pairs of cross-intersecting families,
$$
|\mathcal B| + |\mathcal C| + 1 \leq |\partial^+ \mathcal B| + |\mathcal C| \leq k + 1 + \ell
$$
and
$$
|\mathcal B| + |\mathcal C| + 1 \leq |\mathcal B| + |\partial^+ \mathcal C| \leq k + \ell + 1
$$
follow.
Both of them imply \eqref{eq:3.3}.
What is more, in view of \eqref{eq:2.6}, should we have equality in \eqref{eq:3.3}, $\lambda(\mathcal B) = \lambda(\mathcal C) = 1$ must hold, too.
That is, in case of equality both $\mathcal B$ and $\mathcal C$ consist of consecutive arcs.
\end{proof}

Note that applying \eqref{eq:3.3} for the case $\mathcal B = \mathcal C$, \eqref{eq:3.2} follows.
Moreover, in case of equality $\mathcal B$ must consist of $k$ consecutive arcs of length~$k$.
Equivalently, $\mathcal B$ consists of the $k$ arcs of length $k$ containing a fixed element.

We still owe the reader the proof of \eqref{eq:3.1}.
It is very simple.
Let us apply \eqref{eq:1.1} with $\pi(\mathcal G) = \mathcal A(n, k)$.
Then $|\mathcal A(n, k)| = n$ and by \eqref{eq:3.2}, $|\pi(\mathcal G) \cap \mathcal F| \leq k$.
Thus $|\mathcal F| \leq \frac{k}{n} {n\choose k} = {n - 1\choose k - 1}$. \hfill $\square$

Let us mention that in \cite{FF2} another simple proof of the Erd\H{o}s--Ko--Rado Theorem is given.

\section{Multiply intersecting families}
\label{sec:4}

\begin{definition}
\label{def:4.1}
Let $s, r \geq 2$ be integers and let $\mathcal F \subset 2^{[n]}$.
If $F_1 \cap \ldots \cap F_s \neq \emptyset$ for all $F_1, \ldots, F_s \in \mathcal F$ then $\mathcal F$ is called \emph{$s$-wise intersecting}.
If $F_1 \cup \ldots \cup F_r \neq [n]$ for all $F_1, \ldots, F_r \in \mathcal F$ then $\mathcal F$ is called \emph{$r$-wise union}.
\end{definition}

The following generalization of the Erd\H{o}s--Ko--Rado Theorem was proved in one of my earliest papers.

\setcounter{theorem}{1}
\begin{theorem}[\cite{F76}]
\label{th:4.2}
Let $n, k, s$ be positive integers, $\mathcal F \subset {[n]\choose k}$, $n > s \geq 2$, $sk \leq (s - 1)n$.
If $\mathcal F$ is $s$-wise intersecting then
\beq
\label{eq:4.1}
|\mathcal F| \leq {n - 1\choose k - 1}.
\eeq
\end{theorem}

In \cite{F87} I showed that unless $s = 2$ and $n = 2k$, equality holds in \eqref{eq:4.1} only if
$\mathcal F = \left\{F \in {[n]\choose k} : x \in F \right\}$ for some $x \in [n]$.
The original proof of \eqref{eq:4.1} was motivated by Katona's proof of \eqref{eq:3.1}.
Let us present a different proof here.

\setcounter{definition}{2}
\begin{definition}
\label{def:4.3}
The families $\mathcal F_1, \ldots, \mathcal F_s$ are called cross-intersecting if $F_1 \cap \ldots \cap F_s \neq \emptyset$ for all choices of $F_1 \in \mathcal F_1, \ldots, F_s \in \mathcal F_s$.
Similarly, these families are called cross-union if $F_1\cup \ldots \cup F_s \neq [n]$ for all choices of $F_1 \in \mathcal F_1, \ldots, F_s \in \mathcal F_s$.
\end{definition}

The next statement generalizes Theorem \ref{th:3.4}.

\setcounter{proposition}{3}
\begin{proposition}
\label{prop:4.4}
Suppose that $\ell_1, \ldots, \ell_s$ are positive integers, $s \leq n$, $\ell_i < n$, satisfying $\ell_1 + \ldots + \ell_s \geq n$.
Let $\mathcal B_i \subset \mathcal A(n, \ell_i)$, $1 \leq i \leq s$ and suppose that $\mathcal B_1, \ldots, \mathcal B_s$ are non-empty and cross-union.
Then
\beq
\label{eq:4.2}
|\mathcal B_1| + \ldots + |\mathcal B_s| \leq (n - \ell_1) + \ldots + (n - \ell_s).
\eeq
\end{proposition}

\begin{proof}
First we prove \eqref{eq:4.2} in the special case $\ell_1 + \ldots + \ell_s = n$.
Let $1 \leq p \leq n$ and define the $s$-tuple $\bigl(D_1^{(p)},\ldots, D_s^{(p)}\bigr)$ of arcs as follows
$D_1^{(p)} = A_{\ell_1}(n, p)$, $D_2^{(p)} = A_{\ell_2}(n, p + \ell_1), \ldots, D_s^{(p)} = A_{\ell_s} (n, p + \ell_1 + \ldots + \ell_{s - 1})$.
Then $D_1^{(p)}, \ldots, D_s^{(p)}$ partition $[n]$.
Consequently at least one of $D_i^{(p)} \in \mathcal B_i$, $1 \leq i \leq s$ fails.
The $n$ choices for $p$ provide altogether at least $n = \ell_1 + \ell_2 + \ldots + \ell_s$ missing sets.
Therefore
$|\mathcal B_1| + \ldots + |\mathcal B_s| \leq sn - (\ell_1 + \ell_2 + \ldots + \ell_s)$, equivalent to \eqref{eq:4.2}.

In the general case let us apply induction on $\ell_1 + \ldots + \ell_s$.
Suppose that \eqref{eq:4.2} is proved for $\ell_1 + \ldots + \ell_s = m$ and consider the case $\ell_1 + \ldots + \ell_s = m + 1$.
Since $s \leq n < m + 1$, we may assume by symmetry that $\ell_1 \geq 2$.
By the Circular Kruskal--Katona Theorem
\beq
\label{eq:4.3}
|\partial^- \mathcal B_1| \geq |\mathcal B_1| + 1.
\eeq
The $s$ families $\partial^- \mathcal B_1, \mathcal B_2, \ldots, \mathcal B_s$ are obviously cross-union.
Applying the induction hypothesis yields
$$
|\partial^-\mathcal B_1| + |\mathcal B_2| + \ldots + |\mathcal B_s| \leq sn - \ell_1 - \ell_2 - \ldots - \ell_s + 1.
$$
Using \eqref{eq:4.3}, \eqref{eq:4.2} follows.
\end{proof}

\setcounter{corollary}{4}
\begin{corollary}
\label{cor:4.5}
Let $n, \ell, s$ be positive integers, $n \geq s \geq 2$, $1 \leq \ell \leq n \leq s\ell$.
Suppose that $\mathcal B \subset \mathcal A(n, \ell)$ is $s$-wise union.
Then
\beq
\label{eq:4.4}
|\mathcal B| \leq n - \ell.
\eeq
\end{corollary}

\begin{proof}
Apply \eqref{eq:4.2} with $\mathcal B = \mathcal B_1 = \mathcal B_2 = \ldots = \mathcal B_s$,
$\ell = \ell_1 = \ldots = \ell_s$.
We obtain $s|\mathcal B| \leq s(n - \ell)$ which yields \eqref{eq:4.4}.
\end{proof}

Noting that $\mathcal B$ is $s$-wise union if and only if the family of complements $\mathcal B^c$ is $s$-wise intersecting we get also

\begin{corollary}
\label{cor:4.6}
Let $n > k \geq 1$, $s \geq 2$ be integers, $(s - 1)n \geq sk$.
Suppose that $\mathcal D \subset \mathcal A(n, k)$ is $s$-wise intersecting.
Then
\beq
\label{eq:4.5}
|\mathcal D| \leq k.
\eeq
\end{corollary}

\begin{proof}
Indeed, setting $\mathcal B = \mathcal D^c$, $\ell = n - k$ we may apply Corollary \ref{cor:4.5}.
Then \eqref{eq:4.4} yields $|\mathcal D| = |\mathcal D^c| \leq n - \ell = k$.
\end{proof}

Now to prove Theorem \ref{th:4.2} is a simple application of Lemma \ref{lem:1.1}, it is the same argument as Katona's simple proof of the Erd\H{o}s--Ko--Rado The\-o\-rem.\hfill $\square$

\smallskip
Let me mention that the idea of considering $r$-wise union antichains was suggested by Zsolt Baranyai during a lecture at Katona's seminar.
Baranyai was one of the most talented of Katona's many students.
All of us were very shocked and sad when he died at the age of 29 due to an unfortunate car accident.
It was before he finished writing his PhD dissertation.
Katona took all the trouble and prepared the full text and succeeded in having the PhD title conferred to Baranyai.
To the best of my knowledge he is the only mathematician in Hungary having posthumously received this title.
It was a little consolation to Baranyai's mother who was absolutely crushed by the sudden death of her ingenious son.
To me it showed the goodness and warmth of heart of my adviser, Gyula Katona.

Let me state one of my early results inspired by Baranyai's question.

\setcounter{theorem}{6}
\begin{theorem}[\cite{F76}]
\label{th:4.7}
Let $n, s$ be positive integers, $s \geq 3$.
Suppose that $\mathcal F \subset 2^{[n]}$ is an $s$-wise intersecting antichain satisfying $(s - 1)n \geq s|F|$ for all $F \in \mathcal F$.
Then
\beq
\label{eq:4.6}
\sum_{F \in \mathcal F} \frac{1}{{n - 1\choose |F| - 1}} \leq 1.
\eeq
\end{theorem}

The case $s = 2$ was proved by Greene, Katona and Kleitman \cite{GKK}.
The proof is based on Corollary \ref{cor:4.6} and the following operation discovered by Sperner.

Let $\mathcal F$ be an antichain and $\ell = \min \{|F| : F \in \mathcal F\}$.
Set $\mathcal F^{(\ell)} = \mathcal F \cap {[n]\choose \ell}$.
If $\ell < n$ define $\mathcal F^* = \bigl(\mathcal F\setminus \mathcal F^{(\ell)}\bigr) \cup \partial^+ \mathcal F^{(\ell)}$.
Then $\mathcal F^*$ is an antichain too.

In our case the $s$-wise intersecting property is maintained as well.
Set $k = \left\lfloor\frac{(s - 1)n}{s}\right\rfloor$.
We keep repeating the operation until eventually $\mathcal F^* \subset {[n]\choose k}$.
By Corollary \ref{cor:4.6}, $|\mathcal F^*| \leq {n - 1\choose k - 1}$.

As to \eqref{eq:4.6} it follows by carefully using \eqref{eq:2.8}. \hfill $\square$

\section{The circular Hilton--Milner Theorem}
\label{sec:5}

So-called \emph{stability theorems} are very important and useful in various branches of mathematics.
They show that in certain situations excluding the optimal solutions one can get considerably better bounds.
The first stability theorem in extremal set theory is the Hilton--Milner Theorem.
Let us present it.

\begin{example}
\label{ex:5.1}
Let $n \geq 2k \geq 4$ be integers and define the Hilton--Milner family $\mathcal H(n,k)$.
First define $[a,b] = \{i \in \mathbb N: \, a \leq i \leq b\}$.
Now set
$$
\mathcal H(n, k) = \left\{H \in {[n]\choose k} : 1 \in H, \, H \cap [2, k + 1] \neq \emptyset\right\} \cup \{[2, k + 1]\}.
$$
\end{example}

It is easy to check that $\mathcal H(n, k)$ is intersecting and $|\mathcal H(2k, k)| = {2k - 1\choose k - 1}$.
For $n \geq 2k$ one has
$$
|\mathcal H(n, k)| = {n - 1\choose k - 1} - {n - k - 1\choose k - 1} + 1 < {n - 1\choose k - 1}.
$$
For $k$ fixed and $n \to \infty$:
\beq
\label{eq:5.1}
|\mathcal H(n, k)| = (k + o(1)) {n - 2\choose k - 2},
\eeq
i.e., it is much smaller than the bound ${n - 1\choose k - 1}$ from the Erd\H{o}s--Ko--Rado Theorem.
Recall that a family $\mathcal F$ satisfying $\bigcap\limits_{F \in \mathcal F} F \neq \emptyset$ is called a \emph{star}.

\setcounter{theorem}{1}
\begin{theorem}[Hilton and Milner \cite{HM}]
\label{th:5.2}
Suppose that $\mathcal H \subset {[n]\choose k}$ is intersecting and $\mathcal H$ is not a star.
Then
\beq
\label{eq:5.2}
|\mathcal H| \leq |\mathcal H(n, k)|.
\eeq
Moreover, for $n > 2k$ equality holds in \eqref{eq:5.2} only if $\mathcal H$ is isomorphic to $\mathcal H(n,k)$.
\end{theorem}

The original proof was long and complicated.
By now there are many simpler proofs: \cite{FF1}, \cite{KZ}, \cite{HK}, \cite{F19} to mention a few.

Unfortunately, as pointed out by Katona \cite{K5}, the Katona Circle Method is not sufficient to prove \eqref{eq:5.2}.
The reason is that for $n > n_0(k)$ one has $|\mathcal H(n, k)| \bigm/{n \choose k} < 1/n$.

Nevertheless, we find it worthwhile to prove the Circular Hilton--Milner Theorem.

\setcounter{example}{2}
\begin{example}
\label{ex:5.3}
Let $3(k - 1) \geq n \geq k$ and consider three points $x_1, x_p, x_q$ on Katona's Circle such that $p \leq k$, $q \leq p + k - 1$ and $q + k - 1 > n$.
(Since $n \leq 3(k - 1)$, $p = k$, $q = 2k - 1$ is a possible choice.)
\end{example}

Let us define $\mathcal M_1(n, k)$ as those arcs of length $k$ that contain both $x_1$ and $x_p$.
Similarly, $\mathcal M_2(n, k)$ consists of all the arcs of length $k$ that contain $x_p$ and $x_q$.
Finally, $\mathcal M_3(n, k)$ consists of the arcs of length $k$ containing both $x_q$ and $x_{1}$.
These definitions imply that $\mathcal M_{p, q} := \mathcal M_1(n, k) \cup \mathcal M_2(n, k) \cup \mathcal M_3(n, k)$ is intersecting.
Indeed, $\bigl|M \cap \{x_1, x_p, x_q\}\bigr| \geq 2$ for all $M \in \mathcal M_{p, q}$.
Also
$$
\bigl|\mathcal M_{p, q}\bigr| = 3k - n
$$
is not hard to check.

\setcounter{theorem}{3}
\begin{theorem}[Circular Hilton--Milner Theorem]
\label{th:5.4}
Suppose that $n \geq 2k > 1$.
Let $\mathcal M \subset \mathcal A(n, k)$ be intersecting and assume that $\mathcal M$ is not a star.
Then {\rm (i), (ii)} and {\rm (iii)} hold:

\hspace*{7pt}{\rm (i)} There exist $M_1, M_2, M_3 \in \mathcal M$ satisfying $M_1 \cap M_2 \cap M_3 = \emptyset$.

\hspace*{3.5pt}{\rm (ii)} $k \geq 2$, $n \leq 3(k - 1)$.

{\rm (iii)} $|\mathcal M| \leq 3k - n$.
\end{theorem}

\begin{proof}
To prove (i) we are going to use Helly's Theorem (cf.\ \cite{PA}).
In the case of the plane it says that if $C_1, \ldots, C_m$ are closed convex sets such that $C_1 \cap \ldots \cap C_m = \emptyset$ then there exist $1 \leq i_1 < i_2 < i_3 \leq m$ so that $C_{i_1} \cap C_{i_2} \cap C_{i_3} = \emptyset$.

To associate closed convex sets with the members of $\mathcal M$ let us arrange $x_1, \ldots, x_n$ on a circle and assume that they form a regular $n$-gon in this order.
To an arc $A$ of length $k$ with head $x_u$ and tail $x_v$ we associate the circular cap $C(A)$ bounded by the circular arc joining $x_u$ and $x_v$ and the segment $x_ux_v$.
Let $O$ be the centre of the circle.
Since $k \leq n / 2$, $C(A)$ does not contain $O$.
The following easy facts are important to note.
If $P \in C(A)$ then the intersection (say $R$) of the halfline $OP$ and the circle is contained in $C(A)$.
If $R$ is not a vertex of the regular $n$-gon then it is on an arc joining two neighbouring vertices, both of which are contained in $C(A)$.

Now we are ready to prove (i).
Since $\mathcal M$ is not a star, $\bigcap\limits_{M \in \mathcal M} M = \emptyset$.
We claim that
$$
\bigcap\limits_{M \in \mathcal M} C(M) = \emptyset.
$$
Indeed, otherwise we can fix a point $P$ in the intersection.
By the above geometrical facts we can find a vertex $x_w$ in the intersection as well.
Then $x_w \in M$ for all $M \in \mathcal M$, a contradiction.

Now Helly's Theorem implies the existence of $M_1, M_2, M_3 \in \mathcal M$ satisfying $C(M_1) \cap C(M_2) \cap C(M_3) = \emptyset$.
Since $M_i \subset C(M_i)$, $i = 1,2,3$, $M_1 \cap M_2 \cap M_3 = \emptyset$ follows.

To prove (ii) is easy.
First observe that $M_i \cap M_j \neq \emptyset$, $1 \leq i < j \leq 3$ and $M_1 \cap M_2 \cap M_3 = \emptyset$ imply $M_1 \cup M_2 \cup M_3 = \{x_1, \ldots, x_n\}$.
Computing the size of this union via inclusion-exclusion yields
$$
n = |M_1 \cup M_2 \cup M_3| = \sum_{1 \leq i \leq 3} |M_i| - \sum_{1 \leq i < j \leq 3} |M_i \cap M_j| \leq 3k - 3, \ \text{ proving (ii)}.
$$
To prove (iii) we apply reverse induction on $k$.
If $k = \frac{n}{2}$ then $3k - n = \frac{n}{2} = k$ and (iii) follows from \eqref{eq:3.2}.
If $n$ is odd and $k = \frac{n - 1}{2}$ then $3k - n = k - 1$.
From \eqref{eq:3.2} we only get the upper bound~$k$.
However, from the \emph{proof} of \eqref{eq:3.3} (cf. the remark after the proof) we infer that $|\mathcal M| = k$ can hold only if the intersecting family $\mathcal M$ is a star.
Since this is not the case, $|\mathcal M| \leq k - 1$ follows.

Now we assume $k < \left\lfloor \frac{n}{2}\right\rfloor$ and that (iii) holds for $k + 1$.
Let $M_1, M_2, M_3 \in \mathcal M$ be three members of $\mathcal M$ with $M_1 \cap M_2 \cap M_3$.
We claim that they are in distinct connected components of the graph $G(\mathcal M)$.

Suppose for contradiction, e.g., that $M_1$ and $M_2$ are in the same component.
This means that all $k$-arcs ``between'' them are in $\mathcal M$ too.
Formally, let $M_1 = A_k(x_{t_1})$, $M_2 = A_k(x_{t_2})$, $M_3 = A_k(x_{t_3})$ with $t_1, t_3, t_3$ in counterclockwise order.
Since $\mathcal M$ is intersecting, the tail $x_{t_2 + k - 1} \in M_3$ and the tail $x_{t_3 + k - 1} \in \mathcal M_1$.
Moving from $x_{t_1}$ toward $x_{t_2}$ one by one, there will be $A_k(x_{t_3 + k})$ on the route.
However, $M_3 \cap A_k(x_{t_3 + k}) = \emptyset$, the desired contradiction.

Thus we proved $\lambda(\mathcal M) \geq 3$.
By \eqref{eq:2.6}, $|\partial^+ \mathcal M| \geq |\mathcal M| + 3$ follows.
Obviously $\partial^+\mathcal M$ is intersecting and $\bigcap\limits_{M \in \partial^+ \mathcal M} M = \emptyset$.
By the induction hypothesis
$$
\hspace*{10mm}|\mathcal M| + 3 \leq |\partial^+ \mathcal M| \leq 3(k + 1) - n = 3k - n + 3, \ \text{ proving (iii).} \hspace*{10mm}\qedhere
$$
\end{proof}

\section{Excluded subposets}
\label{sec:6}

Let us first state an important result of Erd\H{o}s, the first generalization of Sperner Theorem.

\begin{definition}
\label{def:6.1}
The $\ell + 1$ distinct sets $R_0, R_1, \ldots, R_\ell$ are called a \emph{chain of length $\ell$} if $R_0 \subset R_1 \subset \ldots \subset R_\ell$.
\end{definition}

\setcounter{example}{1}
\begin{example}
\label{ex:6.2}
Let $n \geq \ell > 0$ and let $0 \leq k_1 < \ldots < k_\ell \leq n$.
Define the Erd\H{o}s family $\mathcal E(k_1, \ldots, k_\ell) = {[n]\choose k_1} \cup \ldots \cup {[n]\choose k_\ell}$.
\end{example}

Obviously, $\mathcal E(k_1, \ldots, k_\ell)$ never contains a chain of length $\ell$.

\setcounter{theorem}{2}
\begin{theorem}[Erd\H{o}s \cite{E1}]
\label{th:6.3}
Let $n \geq \ell > 0$ be integers and suppose that $\mathcal F \subset {[n]\choose k}$ contains no chain of length $\ell$.
Then
\beq
\label{eq:6.1}
|\mathcal F| \leq \max_{0 \leq k_1 < \ldots < k_\ell \leq n} \bigl|\mathcal E(k_1, k_2, \ldots, k_\ell)\bigr|.
\eeq
\end{theorem}

It is easy to see that the maximum of the RHS is attained when $k_1, \ldots, k_\ell$ form an interval with $\frac{n}{2}$ in its centre.

\setcounter{proposition}{3}
\begin{proposition}
\label{prop:6.4}
Let $n \geq \ell > 0$ and suppose that $\mathcal B \subset \mathcal A(n)$ contains no chain of length~$\ell$.
Then
\beq
\label{eq:6.2}
|\mathcal B| \leq \ell n.
\eeq
\end{proposition}

\begin{proof}
Consider the head $H(B)$ for the arcs $B \in \mathcal B$.
Since arcs sharing the same head form a chain, there are at most $\ell$ arcs $B \in \mathcal B$ satisfying $H(B) = x_i$ for each $1 \leq i \leq n$.
\end{proof}

Instead of chains one can consider excluding other configurations defined by inclusions.
Katona initiated this research and made several fundamental contributions.
We refer the interested reader to the excellent book \cite{GP} written by two of his students.

As an illustration we present a nice exact result due to Katona et al.

\setcounter{definition}{4}
\begin{definition}
\label{def:6.5}
Four distinct sets $E, F, G, H$ are said to form a butterfly if both $E$ and $F$ are contained in both $G$ and $H$.
(Equivalently, $E \cup F \subset G \cap H$.)
\end{definition}

Note that we do not exclude $E \subset F$ or $G \subset H$.
Therefore a chain of length $3$ is a butterfly, too.

It is easy to see that the Erd\H{o}s family $\mathcal E(k, k + 1)$ contains no butterfly.

A nice thing about butterflies is that a family $\mathcal F$ contains a butterfly if and only if $\mathcal F^c = \{[n] \setminus F : F \in \mathcal F\}$ contains a butterfly.

From now on we consider families $\mathcal F \subset 2^{[n]}$ which are butterfly-free (contain no butterfly).

\setcounter{claim}{5}
\begin{claim}
\label{cl:6.6}
If\/ $\emptyset, [n] \in \mathcal F$ then $\mathcal F \setminus \{\emptyset, [n]\}$ is an antichain.\hfill $\square$
\end{claim}

For such families $|\mathcal F| \leq 2 + {n\choose \lceil n/2\rceil} < {n\choose \lceil n/2\rceil - 1} + {n\choose \lceil n/2\rceil}$ for $n \geq 3$.
Therefore we assume also that not both $\emptyset$ and $[n]$ are in.

Suppose next that $\emptyset \in \mathcal F$ but $[n] \notin \mathcal F$ (note that in this case $\emptyset \notin \mathcal F^c$ but $[n] \in \mathcal F^c$).
The following argument is due to D. Gerbner.
If the singleton $\{i\}$, $i \in [n]$ is missing from $\mathcal F$ then the family
$(\mathcal F \setminus \{\emptyset\}) \cup \bigl\{\{i\}\bigr\}$ is of the same size and butterfly-free.
If ${[n]\choose 1} \subset \mathcal F$ then $\widetilde{\mathcal F} := \mathcal F \setminus \left({[n]\choose 1} \cup \{\emptyset\}\right)$ consists of pairwise disjoint sets of size at least~$2$.
Indeed, if $G, H \in \widetilde {\mathcal F}$ and $i \in G\cap H$ then $\emptyset, \{i\}, G$ and $H$ form a butterfly.
Thus $|\widetilde{\mathcal F} | \leq \frac{n}{2}$ and $|\mathcal F| \leq 1 + \frac32 n$.
Consequently we may assume that neither $\emptyset$ nor $[n]$ is in~$\mathcal F$.

\setcounter{theorem}{6}
\begin{theorem}[De Bonis, Katona and Swanepoel \cite{BKS}]
\label{th:6.7}
Suppose that $\mathcal F \subset 2^{[n]}$ is butterfly-free with $1 \leq |F| < n$ for all $F \in \mathcal F$.
Then
\beq
\label{eq:6.3}
\sum_{F \in \mathcal F} 1\Bigm/ {n\choose |F|} \leq 2.
\eeq
\end{theorem}

To prove \eqref{eq:6.3} the authors show the circular version.

\begin{theorem}[\cite{BKS}]
\label{th:6.8}
Suppose that $\mathcal B \subset \mathcal A(n)$ is butterfly-free.
Then
\beq
\label{eq:6.4}
|\mathcal B| \leq 2n.
\eeq
\end{theorem}

\begin{proof}
In \cite{BKS} a proof based on counting certain type of full chains in $\mathcal A(n)$ is given.
Let us give a simpler proof.

We divide $\mathcal B$ into three families:
\begin{align*}
\mathcal P &= \bigl\{P \in \mathcal B: \not\exists B \in \mathcal B, \, B \subsetneqq P \bigr\} \ \text{ (minimal sets),}\\
\mathcal R &= \bigl\{R \in \mathcal B: \not\exists B \in \mathcal B, \, R \subsetneqq P \bigr\} \ \text{ (maximal sets),}\\
\mathcal Q &= \mathcal B \setminus (\mathcal P \cup \mathcal R) \ \ \text{ (the rest)}.
\end{align*}
Note that $\mathcal P$ and $\mathcal R$ might overlap and $\mathcal Q$ might be empty.
However, if $Q \in \mathcal Q$ then there exist \emph{exactly} one $P = P(Q)$ and $R = R(Q)$ such that $P \subset Q \subset R$.
Indeed, if for example there were two distinct $P, P' \in \mathcal P$ contained in $Q$ then $P, P'$ and $Q, R(Q)$ form a butterfly.

We should point out that $\mathcal P, \mathcal R$ and $\mathcal Q$ are all antichains.
It is evident for $\mathcal P$ and $\mathcal R$.
As to $\mathcal Q$, should $Q \subset Q'$ hold for distinct $Q, Q' \in \mathcal Q$, the four sets $P(Q), Q, Q', R(Q')$ would form a butterfly.

To prove \eqref{eq:6.4} we prove two inequalities:
\begin{align}
\label{eq:6.5}
|\mathcal P| + |\mathcal Q|/2 & \leq n,\\
\label{eq:6.6}
|\mathcal R| + |\mathcal Q|/2 & \leq n.
\end{align}
Adding these two $|\mathcal F| \leq |\mathcal P| + |\mathcal Q| + |\mathcal R| \leq 2n$ follows.

Noting that for the family of complements $\mathcal B^c$, $\mathcal Q(\mathcal B^c) = \mathcal Q^c$,
$\mathcal P(\mathcal B^c) = \mathcal R^c$, $\mathcal R(\mathcal B^c) = \mathcal P^c$, \eqref{eq:6.5} and \eqref{eq:6.6} are equivalent.
Therefore we only prove \eqref{eq:6.5}.

Consider $Q \in \mathcal Q$ and let $H(Q)$ $(T(Q))$ denote the head (tail) of $Q$, respectively.
Since $Q$ is contained in at most one member of $\mathcal R$, either $H(Q) \neq H(R)$ holds for all $R \in \mathcal R$.
Or $T(Q) \neq T(R)$ for all $R \in \mathcal R$.
By symmetry assume that the former holds for at least $|\mathcal Q|/2$ members of $\mathcal Q$.
Let $H$ be the set of these heads.

Since $\mathcal Q$ is an antichain, $|H| \geq |\mathcal Q|/2$.
Set $H' = \{H(R) : R \in \mathcal R\}$.
By our choice, $H' \cap H = \emptyset$.
Since $\mathcal R$ is an antichain, $|H'| = |\mathcal R|$.
Thus $|\mathcal R| + |\mathcal Q|/2 \leq |H' \cup H| \leq n$ proving \eqref{eq:6.5} and concluding the proof.
\end{proof}

Now \eqref{eq:6.3} follows from \eqref{eq:6.4} in exactly the same way as we proved \eqref{eq:1.2} using \eqref{eq:1.3}.

On the other hand \eqref{eq:6.3} and the considerations for butterfly-free families containing $\emptyset$ or $[n]$ imply:

\begin{theorem}[\cite{BKS}]
\label{th:6.9}
Suppose that $n \geq 3$ and $\mathcal F \subset 2^{[n]}$ is butterfly-free.
Then
\beq
\label{eq:6.7}
|\mathcal F| \leq {n\choose \lfloor n/2\rfloor - 1} + {n\choose \lfloor n/2\rfloor}.
\eeq
\end{theorem}

We should mention that De Bonis, Katona and Swanepoel proved that except for $n = 4$ the optimal families are necessarily of the form $\mathcal E(k_1, k_2)$.

\section{A new proof of a theorem of Hilton}
\label{sec:7}

Hilton proved the following generalization of the Erd\H{o}s--Ko--Rado Theorem.

\begin{theorem}[Hilton \cite{H}]
\label{th:7.1}
Let $n, k, q$ be positive integers, $n \geq 2k$.
Suppose that $\mathcal F_1, \mathcal F_2, \ldots, \mathcal F_q \subset {[n]\choose k}$ are \emph{pairwise} cross-intersecting, that is for all $1\leq i < j \leq s$ and all pairs $F_i \in \mathcal F_i$, $F_j \in \mathcal F_j$, $F_i \cap F_j \neq \emptyset$.
Then
\beq
\label{eq:7.1}
\sum_{1 \leq i \leq q} |\mathcal F_i| \leq \max\left\{{n \choose k}, q{n - 1\choose k - 1}\right\}.
\eeq
\end{theorem}

Before presenting the new proof let us show that \eqref{eq:7.1} is best possible.
Choosing $\mathcal F_1 = {[n]\choose k}$ and $\mathcal F_j = \emptyset$ for $2 \leq j \leq q$ gives ${n\choose k}$.
On the other hand letting $\mathcal F_1 = \ldots = \mathcal F_q = \left\{F \in {[n]\choose k} : 1 \in F \right\}$ gives
$q{n - 1\choose k - 1}$ for the sum $|\mathcal F_1| + \ldots + |\mathcal F_q|$.

To deduce the Erd\H{o}s--Ko--Rado Theorem from \eqref{eq:7.1} fix a $q$ satisfying ${n\choose k} \leq q {n - 1\choose k - 1}$ ($q = \lfloor \frac{n}{k}\rfloor$ is sufficient).
Let $\mathcal F \subset {[n]\choose k}$ be intersecting.
Set $\mathcal F_1 = \ldots = \mathcal F_q = \mathcal F$.
Then $\mathcal F_i$ and $\mathcal F_j$ are cross-intersecting for $1 \leq i < j \leq q$.
Thus \eqref{eq:7.1} implies
$$
q|\mathcal F| \leq q{n - 1\choose k - 1} \ \ \text{ as desired.}
$$

\begin{proof}[Proof of \eqref{eq:7.1}]
First note the following easy fact.
If $\mathcal F_i$ and $\mathcal F_j$ are cross-intersecting then $\mathcal F_i \cup \mathcal F_j$ and $\mathcal F_i \cap \mathcal F_j$ are cross-intersecting as well.
Using this repeatedly for all pairs $(1, j)$, $1 < j \leq q$, that is replacing $\mathcal F_1$ by $\mathcal F_1 \cup \mathcal F_j$ and $\mathcal F_j$ by $\mathcal F_1 \cap \mathcal F_j$, will not alter \eqref{eq:7.1} and maintains the pairwise cross-intersecting property.
Moreover, eventually we obtain $q$ families, which by abuse of notation we still denote by $\mathcal F_1, \ldots, \mathcal F_q$, having the additional property $\mathcal F_1 \supset \mathcal F_j$ for all $1 < j \leq q$.
Consequently, $\mathcal F_j$ and $\mathcal F_j$ are cross-intersecting which is the same as saying that $\mathcal F_j$ is intersecting.

By symmetry assume $|\mathcal F_2| \geq |\mathcal F_3| \geq \ldots \geq |\mathcal F_q|$.

Consequently,
\beq
\label{eq:7.2}
\sum_{1 \leq i \leq q} |\mathcal F_i| \leq |\mathcal F_1| + (q - 1)|\mathcal F_2|.
\eeq

To prove \eqref{eq:7.1} let us show its circular version for the RHS of \eqref{eq:7.2}.

\setcounter{proposition}{1}
\begin{proposition}
\label{prop:7.2}
Let $c > 1$ be a constant.
Suppose that $n \geq 2k$, $\mathcal B_1, \mathcal B_2 \subset \mathcal A(n, k)$ with $\mathcal B_2 \subset \mathcal B_1$, $\mathcal B_1$ and $\mathcal B_2$ cross-intersecting.
Then
\beq
\label{eq:7.3}
|\mathcal B_1| + c|\mathcal B_2| \leq \max \{n, (1 + c)k\}.
\eeq
\end{proposition}

\begin{proof}
If $\mathcal B_2 = \emptyset$ then $|\mathcal B_1| + c|\mathcal B_2| \leq |\mathcal A(n, k)| = n$.
If $\mathcal B_2 \neq \emptyset$ then we may apply \eqref{eq:3.3} to $\mathcal B_1$ and $\mathcal B_2$.
This yields $|\mathcal B_1| + |\mathcal B_2| \leq 2k$.

Applying \eqref{eq:3.2} to $\mathcal B_2$ yields $(c - 1)|\mathcal B_2| \leq (c - 1)k$.
Adding these two inequalities gives \eqref{eq:7.3}.
\end{proof}

Using \eqref{eq:7.3} with $c = q - 1$ and averaging over all cyclic permutations yields
$$
|\mathcal F_1| + (q - 1)|\mathcal F_2| \leq \max \left\{{n\choose k}, q {n - 1\choose k - 1}\right\} \ \ \text{ as desired.}\qquad \square
$$
Let us mention that Borg \cite{Bo} gave another simple proof of Theorem \ref{th:7.1}.

Let us use the results from Section \ref{sec:4} and generalize \eqref{eq:7.1} to $s$-wise cross-intersecting families.

\setcounter{theorem}{2}
\begin{theorem}
\label{th:7.3}
Let $n, k, s, q$ be positive integers, $q \geq s \geq 2$, $(s - 1)n \geq sk$.
Suppose that $\mathcal F_1, \ldots, \mathcal F_q \subset {[n]\choose k}$ are $s$-wise cross-intersecting, that is,
for all possible choices of $1 \leq i_1 < \ldots < i_s \leq q$ the families $\mathcal F_{i_1}, \ldots, \mathcal F_{i_s}$ are cross-intersecting.
Then
\beq
\label{eq:7.4}
\sum_{1 \leq i \leq q} |\mathcal F_i| \leq \max \left\{(s - 1) {n\choose k}, q {n - 1\choose k - 1}\right\}.
\eeq
\end{theorem}

To see that \eqref{eq:7.4} is best possible is easy.
One example being $\mathcal F_1 = \ldots = \mathcal F_{s - 1} = {[n]\choose k}$, $\mathcal F_s = \ldots = \mathcal F_q = \emptyset$.
The other is $\mathcal F_1 = \ldots = \mathcal F_q = \left\{ F \in {[n]\choose k} : 1 \in F\right\}$.

\setcounter{observation}{3}
\begin{observation}
\label{obs:7.4}
For $1 \leq i < j \leq q$ replacing $\mathcal F_i$ and $\mathcal F_j$ by $\mathcal F_i \cup \mathcal F_j$ and $\mathcal F_i \cap \mathcal F_j$ is legitimate.
\end{observation}

\begin{proof}
Indeed, the intersection conditions are maintained and
$$
\hspace*{30mm}|\mathcal F_i \cup \mathcal F_j| + |\mathcal F_i \cap \mathcal F_j| = |\mathcal F_i| + |\mathcal F_j| \hspace*{30mm}\qedhere
$$
\end{proof}

Using this replacement operation repeatedly results in a \emph{nested} system of families, that is, $\mathcal F_1 \supset \mathcal F_2 \supset \ldots \supset \mathcal F_q$.
Therefore in proving \eqref{eq:7.4} we assume that the $\mathcal F_i$ are nested.
This implies that $\mathcal F_\ell$ is $s$-wise intersecting for $s \geq \ell \geq q$.

Let us prove the corresponding statement for the Katona Circle.

\setcounter{proposition}{4}
\begin{proposition}
\label{prop:7.5}
Let $n, k, s$ be integers, $s \geq 2$, $(s - 1) n \geq sk$.
Suppose that $\mathcal B_s \subset \mathcal B_{s - 1} \subset \ldots \subset \mathcal B_1 \subset \mathcal A(n, k)$ are ($s$-wise) cross-intersecting.
Let $c \geq 1$ be a constant.
Then
\beq
\label{eq:7.5}
|\mathcal B_1| + \ldots + |\mathcal B_{s - 1}| + c|\mathcal B_s| \leq \max (s - 1)n, (s - 1 + c)k.
\eeq
\end{proposition}

\begin{proof}
If $\mathcal B_s = \emptyset$ then $|\mathcal B_1| + \ldots + |\mathcal B_s| = |\mathcal B_1| + \ldots + |\mathcal B_{s - 1}| \leq (s - 1)n$.
If $\mathcal B_s \neq \emptyset$ then applying Proposition \ref{prop:4.4} to the complements, $\mathcal B_i^c$, $1 \leq i \leq s$, implies
$$
|\mathcal B_1| + |\mathcal B_2| + \ldots + |\mathcal B_s| \leq sk.
$$
Using nestedness we infer that $\mathcal B_s$ is $s$-wise intersecting.
Thus by Corollary \ref{cor:4.6}
$$
(c - 1)|\mathcal B_s| \leq (c - 1) k.
$$
Adding these two inequalities we obtain \eqref{eq:7.5}.
\end{proof}

To prove \eqref{eq:7.4} let $\pi$ be a random permutation and $\mathcal B_i = \mathcal F_i \cap \mathcal A(n, k)$.
Since the $\mathcal F_i$ are nested, the conditions of Proposition \ref{prop:7.5} are satisfied for $\mathcal B_1, \ldots, \mathcal B_s$.
Also, $|\mathcal B_\ell| \leq |\mathcal B_s|$ for $s \leq \ell \leq q$.
Setting $c = q - s + 1$ we infer
$$
\sum_{1 \leq i \leq q} |\mathcal B_i| \leq \max \bigl\{(s - 1)n, q k\bigr\}.
$$
Now \eqref{eq:7.4} follows by the usual averaging.
\end{proof}

\section{A different way of using Katona's Circle}
\label{sec:8}

Let $n \geq s \geq 1$ be integers.
For a family $\mathcal F \subset 2^{[n]}$ define its matching number $\nu(\mathcal F)$ as
$$
\nu(\mathcal F) = \max \bigl\{s : \,\text{there exist } F_1, \ldots, F_s \in \mathcal F, \ F_i \cap F_j = \emptyset \text{ for } 1 \leq i < j \leq s\bigr\}.
$$
If $\emptyset \in \mathcal F$ then $\nu(\mathcal F) = \infty$ but otherwise $\nu(\mathcal F) \leq n$.
Let us define the function $f(n, s)$ as
$$
f(n, s) = \max \bigl\{|\mathcal F|: \, \mathcal F\subset 2^{[n]}, \nu(\mathcal F) \leq s\bigr\}.
$$

The catch problem for Katona's seminar from the introduction can be stated as $f(n, 1) = 2^{n - 1}$.
For $s \geq 2$ to determine $f(n, s)$ is much more difficult.
As a matter of fact, except for $n$ relatively small, $f(n, s)$ is known only for $n \equiv 0, -1, -2$ $(\text{\rm mod }s + 1)$
(cf. \cite{Kl2}, \cite{Q}, \cite{FK2}).
For the case $n \equiv - 1\ (\text{\rm mod }s + 1)$ Erd\H{o}s proposed the following natural construction.

\begin{example}
\label{ex:8.1}
Let $n \equiv k(s + 1) - 1$ and define the family
$\mathcal D(n, s) = \{D \subset [n]: \, |D| \geq k\}$.
Then $\nu(\mathcal D) = s$ is obvious.

Following the conjecture of Erd\H{o}s, Kleitman \cite{Kl2} proved that
$$
f(k(s + 1) - 1, s) = |\mathcal D(n,s)| = \sum_{i \geq k} {n\choose i}.
$$
\end{example}

In \cite{FK1} we proved the slightly stronger result:

\setcounter{theorem}{1}
\begin{theorem}[Frankl, Kupavskii]
\label{th:8.2}
Suppose that $n = k(s + 1) - 1$, $\mathcal F \subset 2^{[n]}$ and $|\mathcal F| \geq |\mathcal D(n,s)|$ then either $\mathcal F = \mathcal D(n, s)$ or $\mathcal F$ contains $s + 1$ members that \emph{partition}~$[n]$.
\end{theorem}

The main tool of the proof is an inequality which we proved by a slight twist of the Katona Circle.
To make the proof transparent let us present it for the special case $s = 2$.
Let us use the notation $f_\ell = \left|\mathcal F \cap {[n]\choose \ell}\right|$, $0 \leq \ell \leq n$.

\setcounter{proposition}{2}
\begin{proposition}
\label{prop:8.3}
Suppose that $n = 3k - 1$, $k \geq 2$.
Let $\mathcal F \subset 2^{[n]}$ be a family that does not contain three pairwise disjoint sets whose union is $[n]$.
Then
\beq
\label{eq:8.1}
\frac{f_{k - 1}}{{n\choose k - 1}} + \frac{f_k}{{n\choose k}} + \frac{f_{k + 1}}{{n \choose k + 1}} \leq 2.
\eeq
\end{proposition}

To prove \eqref{eq:8.1} we consider a (cyclic) permutation $x_1, x_2, \ldots, x_n, x_1$, the families of arcs $\mathcal A(n, k - 1)$, $\mathcal A(n, k)$ along with a different type of family $\mathcal B(n, k + 1)$ that we define now.

Recall that $A_k(x_i)$ is the arc of length $k$ with head $x_i$ and tail $x_{i + k - 1}$.
Define $B_{k + 1}(x_i) = A_k(x_i) \cup \{x_{i - k}\}$, $1 \leq i \leq n$ and
$\mathcal B(n, k + 1) = \bigl\{B_{k + 1}(x_i): \, 1 \leq i\leq n\bigr\}$.

Note that $n = 3k - 1$ implies $x_{i - k} = x_{i + 2k - 1}$.
Consequently, the three sets $B_{k + 1}(x_i)$, $A_{k - 1} (x_{i - k + 1})$, $A_{k - 1}(x_{i + k})$ form a partition of $[3k - 1]$.

\setcounter{lemma}{3}
\begin{lemma}
\label{lem:8.4}
For every permutation $\pi = (x_1, \ldots, x_n)$
\beq
\label{eq:8.2}
|\mathcal F\cap \mathcal A(n, k - 1)| + |\mathcal F \cap \mathcal A(n, k)| + |\mathcal F \cap \mathcal B(n, k + 1)| \leq 2n.
\eeq
\end{lemma}

Once \eqref{eq:8.2} is proved, \eqref{eq:8.1} follows in exactly the same way as before.
E.g., as \eqref{eq:1.2} followed from \eqref{eq:1.3}.

\begin{proof}[Proof of \eqref{eq:8.2}]
Let us define $R = \bigl\{i : A_{k - 1}(x_i) \in \mathcal F\bigr\}$, $\mathcal S = \mathcal A(n, k) \setminus \mathcal F$, $\mathcal T = \mathcal B(n, k + 1)\setminus \mathcal F$.
With this notation the LHS of \eqref{eq:8.2} is
$$
|R| + (n - |\mathcal S|) + (n - |\mathcal T|) = 2n + |R| - (|\mathcal S| + |\mathcal T|).
$$
That is, we have to prove
\beq
\label{eq:8.3}
|R| \leq |\mathcal S| + |\mathcal T|.
\eeq

To achieve this we are going to define an injection $\varphi : R \to \mathcal S \cup \mathcal T$.
For this purpose let us partition $R$ into $R = R_0 \cup R_1$ where $i \in R_0$ iff $i \in R$ but $i - k \notin R$, $R_1 = R \setminus R_0$.

For $i \in R_1$ consider the three sets $A_{k - 1}(x_i)$, $A_{k - 1}(x_{i - k})$, $B_{k + 1}(x_{i + k})$.
Since these form a partition of $[n]$ and the first two sets are in $\mathcal F$, $B_{k + 1}(x_{i + k}) \notin \mathcal F$.
Define $\varphi(i) = B_{k + 1}(x_{i + k})$.

For $i \in R_0$ consider the three sets $A_{k - 1}(x_i)$, $A_k(x_{i - k})$, $A_k(x_{i + k - 1})$.
Since these form a partition of $[n]$, at least one of the two $k$-arcs is in $\mathcal S$.
Define $\varphi(i) = A_k(x_{i - k})$ if $A_k(x_{i - k}) \notin \mathcal F$ and
$\varphi(i) = A_k(x_{i + k - 1})$ if $A_k(x_{i - k}) \in \mathcal F$ but $A_k(x_{i + k - 1}) \notin \mathcal F$.

To conclude the proof we only have to show that the same member of $\mathcal S$ is never assigned to two distinct $x_i$.

Noting $i + k - 1 \equiv i - 2k$ $(\text{\rm mod }3k - 1)$, we see that if $A_k(x_{i + k - 1})$ is assigned to two $x_j$'s then $j = i$ and $j' = i - k$.
However, $x_{i - k} \notin R$.
That is, nothing is assigned to $x_{i - k}$.
This concludes the proof of \eqref{eq:8.3} and thereby \eqref{eq:8.2}.
\end{proof}

\section{The Erd\H{o}s Matching Conjecture for the circle}
\label{sec:9}

Recall that for a family $\mathcal F \subset 2^{[n]}$ the matching number $\nu(\mathcal F)$ is the maximum number of pairwise disjoint members in $\mathcal F$.
E.g., $\nu(\mathcal F) = 1$ if and only if $\mathcal F$ is non-empty and intersecting.

There are two natural constructions due to Erd\H{o}s \cite{E2} preventing matchings of size $r + 1$.
The first is the complete $k$-graph $\mathcal K(k(r + 1) - 1, k) := {[k(r + 1) - 1]\choose k}$.
The second is $\mathcal L(n, k, r) := \left\{L \in {[n]\choose k} :\, L \cap [r] \neq \emptyset\right\}$.

Let us state the \emph{Erd\H{o}s Matching Conjecture} or EMC for short:

\begin{conjecture}[\cite{E2}]
\label{con:9.1}
\ Suppose that $n, k, r$ are positive integers, $n \geq$
$ k(r + 1)$.
Let $\mathcal F \subset {[n]\choose k}$ satisfy $\nu(\mathcal F) \leq r$.
Then
\beq
\label{eq:9.1}
\nu(\mathcal F) \leq \max \left\{{k(r + 1) - 1\choose k}, {n \choose k} - {n - r\choose k} \right\}.
\eeq
\end{conjecture}

The values on the RHS of \eqref{eq:9.1} correspond to $\bigl|\mathcal K(k(r + 1) - 1, k)\bigr|$ and $|\mathcal L(n, k, r)|$, respectively.
Let us note that \eqref{eq:9.1} is trivial for $r = 1$.
It was proved for $r = 2$ and $3$ by Erd\H{o}s and Gallai \cite{EG} and the present author \cite{F12}, respectively.

\setcounter{theorem}{1}
\begin{theorem}[\cite{F13}]
\label{th:9.2}
Suppose that $n > (2r + 1)k - r$ then \eqref{eq:9.1} is true and up to isomorphism $\mathcal L(n, k, r)$ is the only optimal family.
\end{theorem}

Let us note that for $r$ sufficiently large the same was proved for $n > \frac53 kr$ by Kupavskii and the author \cite{FK2}.

However, to solve EMC for the full range appears to be hopelessly difficult.

For Katona's Circle we can get a complete solution.

\begin{theorem}
\label{th:9.3}
Let $n, k, r$ be positive integers, $n \geq k(r + 1)$.
Suppose that $\mathcal G \subset \mathcal A(n, k)$ and $\nu(\mathcal G) \leq r$.
Then
\beq
\label{eq:9.2}
|\mathcal G| \leq kr.
\eeq
\end{theorem}

Let us note that \eqref{eq:9.2} generlizes the Circular Erd\H{o}s--Ko--Rado Theorem (Theorem \ref{th:3.2}).
The proof is in some sense even simpler.

Let us first prove an easy lemma.

\setcounter{lemma}{3}
\begin{lemma}
\label{lem:9.4}
Suppose that $\mathcal B \subset \mathcal A(n, k)$, $|\mathcal B| = k(r + 1)$.
If $\mathcal G_0 \subset \mathcal B$ satisfies $\nu(\mathcal G_0) \leq r$ then
\beq
\label{eq:9.3}
|\mathcal G_0| \leq kr.
\eeq
\end{lemma}

\begin{proof}
WLOG let $\pi = (1,2,\ldots, n)$.
Let $1 \leq y_1 < y_2 < \ldots < y_{k(r + 1)} \leq n$ in this order be the heads of the arcs in $\mathcal B$.
For $1 \leq j \leq k$ define $\mathcal B_j = \left\{A_k(y_j), A_k(y_{j + k}), \ldots, A_k(y_{j + rk})\right\}$.
It is easy to check that $\mathcal B_j$ consists of $r + 1$ pairwise disjoint $k$-arcs, implying
$$
|\mathcal G_0 \cap \mathcal B_j| \leq r.
$$
Summing this inequality for $1 \leq j \leq k$ yields \eqref{eq:9.3}.
\end{proof}

\begin{proof}[Proof of Theorem \ref{th:9.3}]
Suppose indirectly that $|\mathcal G| > kr$.
Then one can easily choose $\mathcal B \subset \mathcal A(n, k)$, $|\mathcal B| = k(r + 1)$ satisfying $|\mathcal B \cap \mathcal G| > kr$.
Setting $\mathcal G_0 = \mathcal B \cap \mathcal G$ we get a contradiction with Lemma \ref{lem:9.4}.
\end{proof}

We should mention that \eqref{eq:9.2} is best possible.
Namely, $n \geq k(r + 1) \geq kr$ enables us to choose $x_{i_1}, \ldots, x_{i_r}$ on the Katona Circle so that they are pairwise at least $k$ apart.
One possible choice is $\{x_{i_1}, \ldots, x_{i_r}\} = \{1, k + 1, \ldots, (r - 1)k + 1\}$ but there are many more non-isomorphic choices.
Then for a set $T$ one defines $\mathcal B_k(T) = \{B \in \mathcal A(n, k) : B \cap T \neq \emptyset\}$.
Clearly, $\nu(\mathcal B(T)) \leq |T|$.
With the choice $T = \{x_{i_1}, \ldots, x_{i_r}\}$, $|\mathcal B_k(T)| = rk$ holds as well.
Let us also mention that for $k > n/(r + 1)$, $\nu(\mathcal A(n, k)) \leq r$, i.e., the assumption $n \geq k(r + 1)$ is necessary.

In Section \ref{sec:3} we showed that the Erd\H{o}s--Ko--Rado Theorem follows from Katona's circular version of it.
However, in the case $r \geq 2$ this approach does not work.
The reason is that fixing an $r$-set $T \subset [n]$ and setting $\mathcal F_T = \left\{F \in {[n]\choose k}: F \cap T \neq \emptyset\right\}$ the size of $\mathcal F_T \cap \mathcal A(n, k) = \mathcal B_k(T)$ is not always $rk$.
Depending on the permutation $\pi$, it can be quite small.

Unlike for $2^{[n]}$, in the case of $\mathcal A(n)$ one can solve completely the corresponding non-uniform version as well.
To avoid cumbersome notation, let us present only the case $r = 2$.

Set $T_2 = \left\{\lfloor n/2\rfloor, n\right\}$ and define $\mathcal B(T_2) = \bigl\{B \in \mathcal A(n): B \cap T_2 \neq \emptyset\bigr\}$.
Define also $b_k = \bigl|\mathcal A(n, k) \cap \mathcal B(T_2)\bigr|$, $1 \leq k \leq n - 1$.

It is easy to check that $b_k = n$ for $k \geq \frac{n}{2}$ and $b_k = 2k$ for $1 \leq k < \frac{n}{2}$.
Thus
$$
\bigl|\mathcal B(T_2)\bigr| = \begin{cases}
2q^2 + q(q - 1) &\text{ if } \ n = 2q,\\
(2q + 1)q + q(q + 1) &\text{ if } \ n = 2q + 1.
\end{cases}
$$

\setcounter{theorem}{4}
\begin{theorem}
\label{th:9.5}
Suppose that $\mathcal G \subset \mathcal A(n)$, $n \geq 3$ and $\nu(\mathcal G) \leq 2$.
Then
\beq
\label{eq:9.4}
|\mathcal G| \leq \bigl|\mathcal B(T_2)\bigr|.
\eeq
\end{theorem}

The proof of \eqref{eq:9.4} is based on a lemma that we state for the more general case $\nu(\mathcal G) \leq r$.

Let us set $g_k = |\mathcal G \cap \mathcal A(n, k)|$.

\setcounter{lemma}{5}
\begin{lemma}
\label{lem:9.6}
Let $p > r \geq 1$ be integers.
Suppose that $n = k_1 + k_2 + \ldots + k_p$.
Let $\mathcal G \subset \mathcal A(n)$ and suppose that $\nu(\mathcal G) \leq r$.
Then
\beq
\label{eq:9.5}
g_{k_1} + \ldots + g_{k_p} \leq rn.
\eeq
\end{lemma}

\begin{proof}[Proof of \eqref{eq:9.5}]
Fix the $p$ arcs $A_1,\ldots, A_p \subset \mathcal A(n)$ so that $|A_j| = k_j$ for $1 \leq j \leq p$ and these arcs form a partition of $[n] = \{x_1, \ldots, x_n\}$.
Then $\nu(\mathcal G) \leq r$ implies
\beq
\label{eq:9.6}
\bigl|\mathcal G \cap \{A_1, \ldots, A_p\}\bigr| \leq r.
\eeq
Rotating this partition around the circle and summing up the inequalities corresponding to \eqref{eq:9.6} yields \eqref{eq:9.5}.

One can get the same conclusion by choosing $A_1, \ldots, A_p$ uniformly at random.
\end{proof}

\begin{proof}[Proof of \eqref{eq:9.4}]
Let us partition $\mathcal G$ into $\mathcal G_b \cup \mathcal G_m \cup \mathcal G_s$ according the size of $|G|$.
Here $b, m$ and $s$ stand for big, medium and small, respectively.
We set 

\smallskip
$\mathcal G_b\, = \left\{G \in \mathcal G : |G| > \dfrac{n}{2}\right\}$,

\smallskip
$\mathcal G_m \!= \left\{G \in \mathcal G : \dfrac{n}{3} < |G| \leq \frac{n}{2}\right\}$,

\smallskip
$\mathcal G_s\, = \left\{G \in \mathcal G : |G| \leq \dfrac{n}{3}\right\}$.

\smallskip

For $\mathcal G_b$, that is, for $n > k > \frac{n}{2}$ we only use the trivial bound
$$
g_k \leq n.
$$
Since we are comparing $|\mathcal G|$ to $|\mathcal B(T_2)|$ and $b_k = n$ fur such $k$, we do not lose anything.
As $b_k = 2k$ for $1 \leq k \leq \frac{n}{2}$, what we need to prove is
\beq
\label{eq:9.7}
\sum_{1 \leq k \leq \frac{n}{2}} g_k \leq \sum_{1 \leq k \leq \frac{n}{2}} 2k.
\eeq

Our plan is to find a collection $\mathcal P = \{P_1, \ldots, P_w\}$ of pairwise disjoint sets, each of them a subset of $\left[1, \left\lfloor\frac{n}{2}\right\rfloor\right]$ with the following two properties.

\hspace*{3.5pt}(i) \quad $\displaystyle\sum_{x \in P_i} x = n$, \qquad $1 \leq i \leq w$;

(ii) \quad $P_1 \cup \ldots \cup P_w$ \ contains the open interval \ $\left(\frac{n}{3}, \frac{n}{2}\right)$.

Let us first show how this implies \eqref{eq:9.4}.
Using \eqref{eq:9.5} for $P_j$ gives
\beq
\label{eq:9.8}
\sum_{k \in P_j} g_k \leq \sum_{k \in P_j} 2k.
\eeq
Set $P = P_1 \cup \ldots \cup P_w$.
Summing the above inequalities yields
\beq
\label{eq:9.9}
\sum_{k \in P} g_k \leq \sum_{k \in P} 2k.
\eeq

The point is that for $1 \leq k \leq \frac{n}{2}$ and $k \notin P$ either $k \leq \frac{n}{3}$ and $g_k \leq 2k$ holds by \eqref{eq:9.2} or $k = \frac{n}{2}$ and $g_k \leq 2k = n$ again.

Thus adding to \eqref{eq:9.9} individually all these inequalities $g_k \leq 2k$ for $1 \leq k \leq \frac{n}{2}$, $k \notin P$ we obtain \eqref{eq:9.7}.

To construct $\mathcal P$ let us distinguish three cases.

\begin{itemize}
\itemsep=0pt
\leftskip=-10pt
\item[(a)] The length of the interval $\left[\left\lceil\frac{n}{3}\right\rceil, \left\lfloor \frac{n}{2}\right\rfloor\right]$, $2w$ is even.
Let $\mathcal P$ consist of the following $w$ triples:
$$
P_i = \left(\left\lfloor \frac{n}{2}\right\rfloor - 2i, \left\lfloor \frac{n}{2}\right\rfloor - 2i - 1, 4i + 1\right)\, : \, 0 \leq i < w.
$$

By the choice of $w$, $P_{w - 1} = \left(\left\lceil \frac{n}{3}\right\rceil + 1, \left\lceil\frac{n}{3}\right\rceil, 4w - 3\right)$ with
$4w - 3 = n - \left\lceil \frac{n}{3} \right\rceil - \left\lceil\frac{n}{3}\right\rceil - 1 < \frac{n}{3}$.
Thus each $P_i$ consists of three distinct integers and (i), (ii) are fulfilled.

\item[(b)] $n = 2q$ is even and the length of the interval $\left[\left\lceil \frac{n}{3}\right\rceil, \frac{n}{2}\right]$, $2w + 1$ is odd.
    In this case we renounce at $\frac{n}{2}$ and cover the interval $\left[\left\lceil \frac{n}3\right\rceil, \frac{n}{2} - 1\right]$ with the $w$ triples
    $$
    P_i = \left(\frac{n}2 - 2i + 1, \frac{n}{2} - 2i, 4i - 1\right), \ \ \ \ 1 \leq i \leq w.
    $$
    These verify (i) and (ii).
\item[(c)] $n = 2q + 1$ is odd and the length of the interval $\left[\left\lceil \frac{n}3\right\rceil, \left\lfloor \frac{n}{2}\right\rfloor\right]$, $2w - 1$ is odd.
In this case we renounce at $q = \frac{n - 1}{2}$ and cover the interval $\left[\left\lceil\frac{n}{3}\right\rceil, \left\lfloor \frac{n}{2}\right\rfloor - 1\right]$ with the triples
    $$
    P_i = \left(\frac{n - 1}{2} - 2i + 1, \frac{n - 1}{2} - 2i, 4i\right) \ \ \ \ i = 1,\ldots, w - 1.
    $$
Instead of constructing $P_w$ let us use \eqref{eq:9.5} for the choice $n = \frac{n - 1}{2} + \frac{n - 1}{2} + 1$ to deduce
\beq
\label{eq:9.10}
g_{\frac{n - 1}{2}} + \frac12 g_1 \leq 2\frac{n - 1}{2} + 1 = n.
\eeq
\end{itemize}

Now adding the inequalities \eqref{eq:9.8} for $1 \leq j < w$ along with \eqref{eq:9.10} will have each $g_k$ occur with coefficient~$1$
for the range of $\mathcal G_m$, that is, for $\frac{n}{3} < k < \frac{n}{2}$.
To conclude the proof of \eqref{eq:9.7} we add $g_k \leq 2k$ for the missing $k$ with $1 < k < \frac{n}{3}$ along with $\frac12 g_1 \leq \frac12 2 = 1$.
This concludes the proof.
\end{proof}

Let us mention that with a slight effort one can prove uniqueness as well.
E.g., in case (c), to have equality in \eqref{eq:9.7} we must have equality in $\frac12 g_1 = 1$.
That is, $g_1 = 2$.
Letting $\{x_i\}, \{x_{i'}\}$, $1 \leq i < i' \leq n$ be the $1$-element sets in $\mathcal G$,
$\mathcal G \subset\bigl\{G \in \mathcal A(n) : G \cap \{x_i, x_{i'}\} \neq \emptyset\bigr\}$ follows from $\nu(\mathcal G) \leq 2$.
To maximize $|\mathcal G|$ we need to choose $i' - i = \left\lfloor \frac{n}2\right\rfloor$ or $\left\lceil \frac{n}2\right\rceil$.
The proof for the cases $r > 2$ is relying much more on inequalities with fractional coefficients for $g_k$ with $k < n/r$.
It is too messy to include here.

\section{Improved bounds for cross-union families}
\label{sec:10}

Let us recall a result from \cite{FT}.

\begin{theorem}
\label{th:10.1}
Suppose that $\mathcal F_1, \mathcal F_2, \ldots, \mathcal F_r \subset {[n]\choose k}$, $kr \geq n$ and $\mathcal F_1, \ldots, \mathcal F_r$ are cross-union.
Then
\beq
\label{eq:10.1}
\prod_{1 \leq i \leq r} |\mathcal F_i| \leq {n - 1\choose k}^r.
\eeq
\end{theorem}

We should mention that the case $r = 2$ is due to Pyber \cite{P}.
Suppose for a moment that $r \geq 3$, $k(r - 1) \geq n$.
Then one can prove \eqref{eq:10.1} for the triple $(n, k, r)$ using \eqref{eq:10.1} for $(n, k, r - 1)$.
Namely, if $|\mathcal F_1| \times \ldots \times |\mathcal F_r| = 0$ then \eqref{eq:10.1} is evident.
Otherwise for any $j$, $1 \leq j \leq r$, the $r - 1$ families $\mathcal F_i$, $1 \leq i \leq r$, $i\neq j$, must be cross-union.
Therefore
\beq
\label{eq:10.2}
\prod_{1 \leq i \leq r, \,i\neq j} |\mathcal F_i| \leq {n - 1\choose k}^{r - 1}.
\eeq
Taking the product of \eqref{eq:10.2} for all $1 \leq j \leq r$ yields \eqref{eq:10.1}.
However, in the range $\frac{n}{r} \leq k < \frac{n}{r - 1}$ this approach does not work.

The reader must have noticed that the RHS of \eqref{eq:10.1} corresponds to the choice $\mathcal F_1 = \ldots = \mathcal F_r = {[n - 1]\choose k}$. \eqref{eq:10.1} maybe stated as ``the geometric mean of the $|\mathcal F_i|$ is at most ${n - 1\choose k}$''.
For most applications it would be much handier to have \eqref{eq:10.1} for the arithmetic mean.
The problem is that except for the case $n = kr$ it is \emph{not} true.
The standard counterexample is $\mathcal F_1 = \ldots = \mathcal F_{r - 1} = {[n]\choose k}$, $\mathcal F_r = \emptyset$.

The aim of the present section is to use the Katona Circle and prove
\beq
\label{eq:10.3}
\sum_{1 \leq i \leq r} |\mathcal F_i| \leq r {n - 1\choose k}
\eeq
in a wide range of cases.

\begin{theorem}
\label{th:10.2}
Let $n, k, r$ be positive integers, $r \geq 3$.
Suppose that $\mathcal F_1, \ldots, \mathcal F_r \subset{[n]\choose k}$ are \emph{non-empty} and cross-union.
Then \eqref{eq:10.3} holds if
$$
\frac{n}{r - 1} \leq k \leq \frac{r - 1}{r} n.
$$
\end{theorem}

\begin{proof}
We are going to use Katona's Circle.
Let $\mathcal A(n, k)$ be the family of $k$-arcs for a fixed cyclic permutation~$\pi$.
Define $\mathcal G_i = \mathcal F_i \cap \mathcal A(n,k)$.

Since ${n - 1\choose k} = (n - k) {n\choose k} \bigm/n$, \eqref{eq:10.3} follows once we prove
\beq
\label{eq:10.4}
\sum_{1 \leq i \leq r} |\mathcal G_i| \leq r(n - k).
\eeq

Let us define $t$ as the minimal integer such that $tk \geq n$.
Note that $t \leq r - 1$ or equivalently $t + 1 \leq r$.

Let $q$ be the number of non-empty families among the $\mathcal G_i$.
By symmetry assume that $\mathcal G_i = \emptyset$ iff $q < i \leq r$.

If $q \geq t$ then we may apply Proposition \ref{prop:4.4} with $s = q$, $\ell_i = k$ for $1 \leq i \leq q$ and infer
$$
\sum_{1 \leq i \leq r} |\mathcal G_i| = \sum_{1 \leq i \leq q} |\mathcal G_i| \leq q(n - k) \leq r(n - k) \ \text{ proving \eqref{eq:10.4}}.
$$

Suppose next that $q \leq t - 1$.
Then we have the obvious bound
$$
\sum_{1 \leq i \leq r} |\mathcal G_i| \leq (t - 1)n.
$$
To conclude the proof we need to show that the RHS is not larger than $r(n - k)$.
In the case $t = 2$ \
$$
n \leq r(n - k) \ \text{ is equivalent to } \ k \leq (r - 1)n/r.
$$
If $t \geq 3$ then recall $r \geq t + 1$ and prove
$$
(t - 1)n \leq (t + 1) (n - k).
$$
Equivalently,
$$
k \leq \frac2{t + 1} n.
$$
Since for $t \geq 3$, $\frac1{t - 1} \leq \frac2{t + 1}$, the above inequality follows from $k < \frac{n}{t - 1}$, the definition of~$t$.
This concludes the proof.
\end{proof}

Let us note that the bound $k \leq \frac{r - 1}{r}n$ is asymptotically necessary.
Namely, if $\frac{r - 1}{r} < \alpha < 1$, $k = \lfloor \alpha n \rfloor$ one can define
$$
\mathcal H_1 = \left\{H \in {[n]\choose k} : [k + 1, n] \not\subset H\right\}, \ \ \ \mathcal H_2 = \ldots = \mathcal H_r = \{[k]\}.
$$
These $r$ families are cross-union and
$$
\lim_{n \to \infty} |\mathcal H_1| \Bigm/{n - 1\choose k} = \frac{n}{n - k} = \frac1{1 - \alpha} > r.
$$

\section{Intersecting-union families on the Katona Circle}
\label{sec:11}

A family $\mathcal F \subset 2^{[n]}$ is called \emph{intersecting-union} or for short an \emph{IU-family} if it is both intersecting and union.
That is for all $F, G \in \mathcal F$, neither $F \cap G = \emptyset$ nor $F \cup G = [n]$ holds.

\begin{theorem}
\label{th:11.1}
Suppose that $\mathcal F \subset 2^{[n]}$ is an IU-family.
Then
\beq
\label{eq:11.1}
|\mathcal F| \leq 2^{n - 2}.
\eeq
\end{theorem}

This result was proved by several authors around 50 years ago (cf. \cite{DL}, \cite{Se}).
One can easily deduce it using the Harris--Kleitman Correlation Lemma (cf. \cite{Ha}, \cite{Kl1}) as well.

It does not appear to be a likely candidate for the Katona Circle Method because the many examples giving equality in \eqref{eq:11.1} are irregular with respect to the circle: $|\mathcal F \cap \mathcal A(n)|$ is strongly dependent on the particular (cyclic) permutation that we choose.

Nevertheless we can solve the corresponding problem for the Katona Circle.

\setcounter{example}{1}
\begin{example}
\label{ex:11.2}
Let $\pi = (x_1, \ldots, x_n)$ be a (cyclic) permutation and fix the integer $r$, $1 < r \leq n$.
Define $\mathcal D(1, r) = \{B \in \mathcal A(n) : 1 \in B, r \notin B\}$.
\end{example}

Obviously, $\mathcal D(1, r)$ is an IU-family.
Let us compute its size.
The arc with head $x_i$ and tail $x_j$ is in $\mathcal D(1,r)$ iff $x_i$ is in $\{x_{r + 1}, x_{r + 2}, \ldots, x_1\}$ and $x_j \in \{x_1, \ldots, x_{r - 1}\}$.
Thus
$$
|\mathcal D(1, r)| = (r - 1)(n - r + 1) \leq \left\lfloor \frac{n}{2}\right\rfloor \left\lceil \frac{n}{2}\right\rceil.
$$

\setcounter{theorem}{2}
\begin{theorem}
\label{th:11.3}
Suppose that $\mathcal D \subset \mathcal A(n)$ is an IU-family.
Then
\beq
\label{eq:11.2}
|\mathcal D| \leq \left\lfloor \frac{n}{2}\right\rfloor \times \left\lceil\frac{n}{2}\right\rceil.
\eeq
\end{theorem}

\begin{proof}
Consider $\mathcal D_k = |\mathcal D \cap \mathcal A(n, k)|$.
We claim that
\begin{align}
\label{eq:11.3}
|\mathcal D_k| &\leq k \ \ \text{ for } \ \ k \leq \frac{n}{2} \ \ \text{ and}\\
\label{eq:11.4}
|\mathcal D_k| &\leq n - k \ \text{ for }\ \ k > \frac{n}{2}.
\end{align}
Indeed \eqref{eq:11.3} follows from the fact that $\mathcal D_k \subset \mathcal D$ is intersecting (cf. \eqref{eq:3.2}).
On the other hand if $k > \frac{n}{2}$ then $n - k < \frac{n}{2}$.
Since $\mathcal D$ has the union property, the family of complements $\mathcal D^c = \{[n] \setminus D : D \in \mathcal D\}$ is intersecting.
This implies \eqref{eq:11.4}.

To prove \eqref{eq:11.2} we simply sum \eqref{eq:11.3} for $1 \leq k \leq \frac{n}{2}$ and \eqref{eq:11.4} for $\frac{n}{2} < k < n$.
This yields
\beq
\label{eq:11.5}
|\mathcal D|\! \leq \! \left(1 + 2 +\! \ldots \! + \! \left\lfloor \frac{n}{2}\right\rfloor\right)\! +\! \left(\left\lfloor \frac{n - 1}{2}\right\rfloor\! +\! \ldots\! +\! 1\right) \! =\!
{\left\lfloor\frac{n}{2}\right\rfloor + 1\choose 2} + {\left\lfloor\frac{n + 1}{2}\right\rfloor\choose 2}.
\eeq
For $n = 2q + 1$ the RHS is $2{q + 1\choose 2} = (q + 1)q$ and for $n = 2q$ it is ${q + 1\choose 2} + {q\choose 2} = q^2$ proving \eqref{eq:11.2}.
\end{proof}

Unlike for IU-families in $2^{[n]}$ the optimal families are unique.
Namely, to have equality in \eqref{eq:11.5} one must have equality in \eqref{eq:11.3} and \eqref{eq:11.4} for \emph{all}~$k$.
In particular $|\mathcal D_1| = 1$, $|\mathcal D_{n - 1}| = 1$.
This implies that $\mathcal D = \mathcal D(i, j)$ for some appropriate $i \neq j$.
Now $|\mathcal D(i,j)| = \left\lfloor \frac{n}{2}\right\rfloor \times \left\lceil \frac{n}{2}\right\rceil$ implies $|i - j| = \left\lfloor \frac{n}{2}\right\rfloor$ or $\left\lceil \frac{n}{2}\right\rfloor$, i.e., $x_i$ and $x_j$ are antipodal for $n$ even and almost antipodal for $n$ odd.

Recalling that the expectation of $|\mathcal F^{(k)} \cap \mathcal A(n, k)|$ is $|\mathcal F^{(k)}| n \bigm/{n\choose k}$, \eqref{eq:11.2} implies
\beq
\label{eq:11.6}
\sum_{F \in \mathcal F} \frac1{{n\choose |F|}} \leq \frac{\left\lfloor\frac{n}{2}\right\rfloor \cdot \left\lceil\frac{n}{2}\right\rceil}{n} \leq \frac{n}{4}.
\eeq
However, this appears to be far from best possible.
We conjecture that $6$ is the correct denominator.

\setcounter{conjecture}{3}
\begin{conjecture}
\label{con:11.4}
Suppose that $\mathcal F \subset 2^{[n]}$ is an IU-family.
Then
\beq
\label{eq:11.7}
\sum_{F \in \mathcal F} 1\Bigm/{n\choose |F|} \leq \frac{n + 1}{6}.
\eeq
\end{conjecture}

Let us mention in connection an old unsolved problem related to IU-families.

\begin{conjecture}[Gronau \cite{Gr}]
\label{con:11.5}
Suppose that $\mathcal F \subset 2^{[n]}$ is a family with the property that for every $F, G \in \mathcal F$ either $F = [n]\setminus G$ or both $F \cap G \neq \emptyset$ and $F \cup G \neq [n]$ are satisfied.
Then
\beq
\label{eq:11.8}
|\mathcal F| \leq 2^{n - 2} \ \ \text{ or } \ n \ \text{ is even and } \ |\mathcal F| \leq {n\choose n/2}.
\eeq
\end{conjecture}

Needless to say, the second case in \eqref{eq:11.8} corresponds to ${[n]\choose n/2}$.
As a matter of fact, $2^{n - 2} > {n\choose n/2}$ for $n \geq 10$.
Thus for $n \geq 9$ the conjecture is that $|\mathcal F| \leq 2^{n - 2}$.
That is, allowing complementary pairs does not alter the veracity of the bound \eqref{eq:11.1}.

To solve Gronau's problem for the Katona Circle is easier.

\setcounter{theorem}{5}
\begin{theorem}
\label{th:11.6}
Suppose that $\mathcal D \subset \mathcal A(n)$, $n \geq 3$ and for all $C, D \in \mathcal B$ either $C \cap D \neq \emptyset$ and $C \cup D \neq [n]$ or $C = [n]\setminus D$.
Then
\beq
\label{eq:11.9}
|\mathcal B| \leq \left\lfloor\frac{n}{2}\right\rfloor \times \left\lceil \frac{n}{2}\right\rceil.
\eeq
\end{theorem}

\begin{proof}
Let us recall the notation $\mathcal D_k = \mathcal B \cap \mathcal A(n, k)$.
Note that for $k \neq n/2$, $\mathcal D_k$ is an IU-family.
Thus \eqref{eq:11.3} and \eqref{eq:11.4} hold.
If $n$ is odd, summing up these inequalities yields \eqref{eq:11.9}.

From now on let $n = 2q$, even.
If $|\mathcal D_q| \leq q$, then we obtain \eqref{eq:11.9} again.
Let $|\mathcal D_q| = q + t$ for some $1 \leq t \leq q$.
Since $|\mathcal A(2q, q)| = 2q$, we infer that there are at least $t$ pairs $(B_i, C_i)$, $1 \leq i \leq t$ such that $B_i = [n] \setminus C_i$.
If $t = q$ then $\mathcal D_q = \mathcal A(n, q)$ and we easily see that no more arcs can be added.
I.e., $\mathcal D = \mathcal A(n, q)$ and 
$|\mathcal D| \leq \left\lfloor \frac{n}{2}\right\rfloor \times \left\lceil \frac{n}{2}\right\rceil$ (with strict inequality unless $n = 4$) follow. 
Thus we may assume $t < q$.

\setcounter{lemma}{6}
\begin{lemma}
\label{lem:11.7}
$\mathcal D_k = \emptyset = \mathcal D_{n - k}$ for $1 \leq k \leq t$.
\end{lemma}

\begin{proof}
Take an arbitrary $D \in \mathcal D$, $|D| = k < q$.
Then $D \cap B_i \neq \emptyset \neq D \cap C_i$ implies that $D$ contains one of the two $2$-arcs formed by the head of $B_i$ and the tail of $C_i$ or by the head of $C_i$ and the tail of $B_i$.
Let $A_i$ be the corresponding one.
For $i = 1, \ldots, t$ we obtain $t$ distinct $2$-arcs $A_i$ and
$|A_1 \cup \ldots \cup A_t| \geq t + 1$.
Since $A_1 \cup \ldots \cup A_t \subset D$ the statement follows.
If $k > q$ then we consider the family $\mathcal D^c = \{[n]\setminus D : D \in \mathcal D\}$ and repeat the above argument.
\end{proof}

Now to conclude the proof of \eqref{eq:11.9} is easy.
$$
|\mathcal D| = \sum_{1 \leq j < n} |\mathcal D_k| = q + t + \sum_{t < k < q} k + \sum_{q < k < 2q - t} (n - k) = q^2 + t - 2{t + 1\choose 2} < q^2. \qedhere
$$
\end{proof}

This approach worked for the Katona Circle because the total number of $k$-arcs is $n$, i.e., it is independent of~$k$.
However, for $2^{[n]}$ the middle binomial coefficient is much larger than most of the others.
Thus one needs some different ideas to prove \eqref{eq:11.8}.

\section{Concluding remarks}
\label{sec:12}

A couple of months ago I decided to write a survey paper to celebrate the eightieth birthday of Gyula Katona, my ex-teacher and one of my best friends.
His numerous contributions helped to make extremal finite set theory from a loose collection of results into a real theory.
Hoping to increase the number of readers I chose as a topic the Katona Circle because of its simplicity and elegance.

To start the work I asked Gyula to send me his work concerning the Circle Method.
Reading those papers I got a lot of inspiration and gradually proved and reproved several results that transformed the survey into a collection of new proofs and results.

This convinced me that may Gyula -- as I strongly wish -- live to be hundred and more, the Katona Circle will survive him!

\end{document}